\documentclass[11pt,a4paper]{article}
\usepackage{}
\usepackage{amsfonts}
\usepackage{mathrsfs}
\usepackage{multirow}
 \usepackage{epsfig}
 \usepackage{epstopdf}
\usepackage[T1]{fontenc}
\usepackage{geometry}
\usepackage{amsbsy,amsmath,latexsym,amsfonts, epsfig, color, authblk, amssymb, graphics, bm}
\usepackage{epsf,slidesec,epic,eepic}
\usepackage{fancybox}
\usepackage{algorithm}
\usepackage{algorithmic}
\usepackage{fancyhdr}
\usepackage{setspace}
\usepackage{nccmath}
\usepackage{cases}
\usepackage{cite}
\usepackage[colorlinks, citecolor=blue]{hyperref}

\newtheorem{theorem}{Theorem}[section]
\newtheorem{lemma}{Lemma}[section]

\newtheorem{definition}{Definition}[section]
\newtheorem{example}{Example}[section]

\newtheorem{corollary}{Corollary}[section]
\newtheorem{assumption}{Assumption}[section]

\newtheorem{remark}{Remark}[section]

\newtheorem{alemma}{Lemma}

\newenvironment{aproof}{{\noindent}}{\hfill$\Box$\medskip}
\newenvironment{proof}{{\noindent \bf Proof:}}{\hfill$\Box$\medskip}

\definecolor{lred}{rgb}{1,0.8,0.8}
\definecolor{lblue}{rgb}{0.8,0.8,1}
\definecolor{dred}{rgb}{0.6,0,0}
\definecolor{dblue}{rgb}{0,0,0.5}
\definecolor{dgreen}{rgb}{0,0.5,0.5}

 \title{Error bound of critical points and KL property of exponent $1/2$
 for squared F-norm regularized factorization}

 \author{Ting Tao\footnote{mattao@mail.scut.edu.cn, School of Mathematics, South China University of Technology.},
 \ Shaohua Pan\footnote{shhpan@scut.edu.cn, School of Mathematics, South China University of Technology, Guangzhou.}
  \ \ {\rm and}\ \
  Shujun Bi\footnote{bishj@scut.edu.cn, School of Mathematics, South China University of Technology, Guangzhou.}}

\begin{document}

 \maketitle

 \begin{abstract}
  This paper is concerned with the squared F(robenius)-norm regularized factorization
  form for noisy low-rank matrix recovery problems. Under a suitable assumption
  on the restricted condition number of the Hessian for the loss function,
  we derive an error bound to the true matrix for the non-strict critical points
  with rank not more than that of the true matrix. Then, for the squared F-norm
  regularized factorized least squares loss function, under the noisy and
  full sample setting we establish its KL property of exponent $1/2$ on its
  global minimizer set, and under the noisy and partial sample setting
  achieve this property for a class of critical points. These theoretical findings
  are also confirmed by solving the squared F-norm regularized factorization problem
  with an accelerated alternating minimization method.
 \end{abstract}


 \medskip
 \section{Introduction}\label{sec1}

  Low-rank matrix recovery problems aim at recovering a true but unknown low-rank
  matrix $M^*\!\in\mathbb{R}^{n\times m}$ from as few observations as possible,
  and have wide applications in a host of fields such as statistics,
  control and system identification, signal and image processing, machine learning,
  quantum state tomography, and so on (see, e.g., \cite{Fazel02,Davenport16,GroLFBE10,Recht10}).
  Generally, when a tight upper bound, say an integer $r\ge 1$,
  is available for the rank $r^*$ of $M^*$, they can be formulated as
  the rank constrained optimization problem
  \begin{equation}\label{rank-constr}
   \min_{X\in \mathbb{R}^{n\times m}}\big\{f(X)\ \ {\rm s.t.}\ \
   {\rm rank}(X)\le r\big\}
  \end{equation}
  where $f\!:\mathbb{R}^{n\times m}\rightarrow \mathbb{R}_{+}$ is
  an empirical loss function. Otherwise, one needs to solve a sequence of rank constrained
  optimization problems with an adjusted upper estimation for $r^*$.
  For the latter scenario, one may consider the rank regularized model
  \begin{equation}\label{rank-reg}
   \min_{X\in \mathbb{R}^{n\times m}}\big\{f(X)+\lambda\,{\rm rank}(X)\big\}
  \end{equation}
  with an appropriate $\lambda>0$ to achieve a desirable low-rank solution.
  Model \eqref{rank-constr}-\eqref{rank-reg} reduce to the rank
  constrained and regularized least squares problem, respectively, when
  \begin{equation}\label{Lsquare}
   f(X):=\frac{1}{2}\|\mathcal{A}(X)-y\|^2\quad\ \forall X\in\mathbb{R}^{n\times m},
  \end{equation}
 where $\mathcal{A}\!:\mathbb{R}^{n\times m}\rightarrow\mathbb{R}^p$ is
 the sampling operator and $y$ is the noisy observation from
 \begin{equation}\label{observe-model}
    y=\mathcal{A}(M^*)+\omega.
 \end{equation}

  Due to the combinatorial property of the rank function, rank optimization
  problems are NP-hard and it is impossible to seek a global optimal solution
  with a polynomial-time algorithm. A common way to deal with them is to adopt
  the convex relaxation technique. For the rank regularized problem \eqref{rank-reg},
  the popular nuclear norm relaxation method (see, e.g., \cite{Fazel02,Recht10,Candes09})
  yields a desirable solution via a single convex minimization problem
  \begin{align}\label{Nuclear-norm}
  \min_{X\in \mathbb{R}^{n\times m}}\big\{f(X)+\lambda\|X\|_*\big\}.
  \end{align}
  Over the past decade active research, this method has made great progress
  in theory (see, e.g., \cite{Candes09,Candes11,Recht10,Negahban11}). Despite its
  favorable performance in theory, improving computational efficiency remains
  a challenge. In fact, almost all convex relaxation algorithms for \eqref{rank-reg}
  require an economic SVD of a full matrix in each iteration, which forms the major
  computational bottleneck and restricts their scalability to large-scale problems.
  Inspired by this, recent years have witnessed the renewed interest in
  the Burer-Monteiro factorization \cite{Burer03} for low-rank matrix
  optimization problems. By replacing $X$ with its factored form
  $UV^\mathbb{T}$ for $(U,V)\in\!\mathbb{R}^{n\times r}\times\mathbb{R}^{n\times r}$
  with $r^*\le r<\min(n,m)$, the factored form of \eqref{Nuclear-norm} is
  \begin{equation}\label{Phi-lam}
  \min_{U\in \mathbb{R}^{n\times r},V\in \mathbb{R}^{m\times r}}
  \Big\{\Phi_{\lambda}(U,V):=f(UV^\mathbb{T})+\frac{\lambda}{2}\big(\|U\|_F^2+\|V\|_F^2\big)\Big\}.
  \end{equation}

  Although the factorization form tremendously reduces the number of
  optimization variables since $r$ is usually much smaller than $\min(n,m)$,
  the intrinsic bi-linearity makes the factored objective functions nonconvex
  and introduces additional critical points that are not global optimizers of
  factored optimization problems. A research line for factored
  optimization problems focuses on the nonconvex geometry landscape,
  especially the strict saddle property
  (see, e.g., \cite{ChiGe17,Ge15,Ge17,Li18,Zhu181,Zhu182,Park17,Sun16,LiWang19,Lee19}).
  Most of these works center around the factorized forms of problem \eqref{rank-constr}
  or their regularized forms with a balanced term except \cite{Li18} in which,
  under a restricted well-conditioned assumption on $f$, the authors proved that
  every critical point $(\overline{U},\overline{V})$ of problem \eqref{Phi-lam}
  with $r\ge{\rm rank}(X^*)$, where $X^*$ is an optimal solution of \eqref{Nuclear-norm},
  either corresponds to a factorization of $X^*$ (i.e., $X^*=\overline{U}\overline{V}^{\mathbb{T}}$)
  or is a strict saddle point (i.e., the critical point at which the Hessian matrix
  has a strictly negative eigenvalue). This, along with the equivalence between \eqref{Nuclear-norm}
  and \eqref{Phi-lam} (see Lemma \ref{lemma2} in Appendix C) and the result of \cite{Lee19},
  implies that many local search algorithms such as gradient descent and its variants
  can find a global optimal solution of \eqref{Phi-lam} with a high probability
  if the parameter $\lambda$ is chosen such that \eqref{Nuclear-norm} has a solution
  with rank at most $r$. Another research line considers the (regularized) factorizations
  of rank optimization problems from a local view, and aims to characterize the convergence rate
  of the iterates in terms of a certain measure or the growth behavior of objective functions around the set of global
  optimal solutions (see, e.g., \cite{Jain13,Park18,SunLuo16,Tu16,ZhangSo18,Zhao15,Zheng16}).

  For problem \eqref{rank-constr} associated to noisy low-rank matrix recovery,
  some researchers are also interested in the error bound to the true $M^*$
  for the local optimizers of the factorization or its regularized form with
  a balanced term. For example, for the noisy low-rank positive semidefinite (PSD) matrix recovery,
  Bhojanapalli et al. \cite{Bhojanapalli16} achieved the error bound for any local
  optimizer of the factorized exact or over-parameterization (i.e., $r\!\ge r^*$)
  under a RIP condition on the sampling operator; and for a general noisy low-rank matrix recovery,
  Zhang et al. \cite{Zhang18} established the error bound for any local optimizer
  of the factorized exact-parameterization with a balanced regularization term
  under a restricted strong convexity and smoothness condition of the loss function.
  However, there are few works to discuss the error bounds for the critical point of
  the factorization associated to the rank regularized problem \eqref{rank-reg}
  or its convex relaxation \eqref{Nuclear-norm} except \cite{Chen19} in which,
  for noisy matrix completion, the nonconvex Burer-Monteiro approach is used to
  show that the convex relaxation approach achieves near-optimal estimation errors.

  This work is concerned with the error bound for the critical point of
  the nonconvex factorization \eqref{Phi-lam} with $r\ge r^*$ and
  the KL property of exponent $1/2$ of $\Phi_{\lambda}$ associated to
  $r=r^*$ under the noisy setting. Specifically, under a suitable assumption on
  the restricted condition number of the Hessian matrix $\nabla^2f$,
  we derive an error bound to the true $M^*$ for those non-strict critical points
  with rank at most $r^*$, which is demonstrated to be optimal in terms of the exact
  characterization of global optimal solutions under the ideal noiseless
  and full sampling setting (see \cite{ZhangSo18}). Different from \cite{Chen19},
  our error bound result is obtained for a general smooth loss by adopting
  a deterministic rather than probability analysis technique. In addition,
  for the least squares loss, under the noisy and full sample setting
  we establish the KL property of exponent $1/2$ of $\Phi_{\lambda}$ associated to
  almost all $\lambda>0$ over its global minimizer set, and under the noisy and
  partial sampling setup, achieve this property of $\Phi_{\lambda}$ only for
  a class of critical points. This extends the result of \cite[Theorem 2]{ZhangSo18}
  to the noisy setting. Together with the strict saddle property of $\Phi_{\lambda}$
  in \cite{Li18} and the result of \cite{Lee19}, this means that for problem \eqref{Phi-lam}
  with the least squares loss in the full sampling, many first-order methods can
  find a global optimal solution in a linear convergence rate
  with a high probability, provided that the parameter $\lambda$ is chosen such that
  \eqref{Nuclear-norm} has an optimal solution with rank at most $r$. Hence,
  it partly improve the convergence analysis results of the alternating minimization methods
  proposed in \cite{Recht10,Hastie15} for solving this class of problems.
  Li et al. \cite{Li18} ever mentioned that the explicit convergence rate
  for some algorithms in \cite{Ge15,Sun16} can be obtained by extending
  the strict saddle property with the similar analysis in \cite{Zhu182},
  but to the best of our knowledge, there is no strict proof for this, and
  the analysis in \cite{Zhu182} is tailored to the factorization form
  with a balanced term.

  \medskip
  \noindent
  {\bf\large Notation:} Throughout this paper, $\mathbb{R}^{n\times m}$ represents the vector space of
  all $n\times m$ real matrices, equipped with the trace inner product
  $\langle X,Y\rangle={\rm trace}(X^{\mathbb{T}}Y)$
  for $X,Y\in\mathbb{R}^{n\times m}$ and its induced F-norm,
  and we stipulate $n\le m$. The notation $\mathbb{O}^{n\times \kappa}$
  denotes the set of $n\times\kappa$ matrices with orthonormal columns,
  and $\mathbb{O}^{n}$ stands for $\mathbb{O}^{n\times n}$.
  Let $I$ denote an identity matrix whose dimension is known from the context.
  For a matrix $X\in\mathbb{R}^{n\times m}$,
  we denote by $\sigma(X)$ the singular value vector of $X$ arranged in a nonincreasing order,
  and by $\sigma^{\kappa}(X)$ for an integer $\kappa\ge 1$ the vector consisting of
  the first $\kappa$ entries of $\sigma(X)$, and write $X_+:=\max(0,X)$ and
  \(
    \mathbb{O}^{n,m}(X):=\big\{(P,Q)\in\mathbb{O}^{n}\times\mathbb{O}^{m}\ |\
    X=P{\rm Diag}(\sigma(X))Q^{\mathbb{T}}\big\},
  \)
  where ${\rm Diag}(z)$ represents a rectangular diagonal matrix with $z$
  as the diagonal vector. We denote by $\|X\|$ and $\|X\|_*$  the spectral norm
  and the nuclear norm of $X$, respectively, by $X^\dag$ the pseudo-inverse of $X$,
  and by ${\rm col}(X)$ the column space of $X$. Let $\mathcal{P}_{\rm on}$ and
  $\mathcal{P}_{\rm off}$ denote the linear mapping from $\mathbb{R}^{(n+m)\times(n+m)}$
  to itself, respectively, defined by
  \begin{equation}\label{MP-on-off}
   \mathcal{P}_{\rm on}
   \Big(\left[\begin{matrix}
         A_{11} & A_{12} \\
         A_{21}& A_{22}
     \end{matrix}\right]\Big)
   =\left[\begin{matrix}
     A_{11} & 0 \\
     0 & A_{22} \\
     \end{matrix}\right],\,
   \mathcal{P}_{\rm off}
   \Big(\left[\begin{matrix}
         A_{11} & A_{12} \\
         A_{21}& A_{22}
     \end{matrix}\right]\Big)
   =\left[\begin{matrix}
     0 & A_{12} \\
     A_{21} & 0 \\
     \end{matrix}\right]
  \end{equation}
  where $A_{11}\in\mathbb{R}^{n\times n}$ and $A_{22}\in\mathbb{R}^{m\times m}$.
  For a given matrix pair $(U,V)\in\mathbb{R}^{n\times\kappa}\times\mathbb{R}^{m\times\kappa}$,
  we denote by $\mathbb{B}((U,V),\delta)$ the closed F-norm ball of radius $\delta$
  centered at $(U,V)$, and by ${\rm dist}((U,V),\Gamma)$ the F-norm distance of $(U,V)$
  from a set $\Gamma\subseteq\mathbb{R}^{n\times\kappa}\times\mathbb{R}^{m\times\kappa}$,
  and write
  \begin{equation}\label{W-UV}
    W=(U;V)\in\mathbb{R}^{(n+m)\times\kappa}\ \ {\rm and}\ \ \widehat{W}=(U;-V)\in\mathbb{R}^{(n+m)\times\kappa}.
  \end{equation}
  We denote by $M^*$ the true matrix of rank $r^*$ with $4r^*\le\min(m,n)$, and write
  \[
   \mathcal{E}^*:=\Big\{\big(P_1^*{\Sigma^*}^{1/2}R,Q_1^*{\Sigma^*}^{1/2}R\big)\,|\, (P^*,Q^*)\in\mathbb{O}^{n,m}(M^*),R\in\mathbb{O}^{r^*}\Big\}
  \]
  where $\Sigma^*={\rm Diag}(\sigma_{1}(M^*),\ldots,\sigma_{r^*}(M^*))\in\mathbb{R}^{r^*\times r^*}$,
  and $P_1^*$ and $Q_1^*$ are the matrix consisting of the first $r^*$ columns of $P^*$ and $Q^*$, respectively.
  For a given $U\in\mathbb{R}^{n\times\kappa}$, we also write $U=(U_1;U_2)$
  where $U_1$ and $U_2$ are the matrix consisting of the first $r^*$ rows
  and the rest $n-r^*$ rows of $U$.   For any given $A,B\in\mathbb{R}^{n\times\kappa}$, define
  \(
    {\rm dist}(A,B):=\mathop{\min}_{R\in\mathbb{O}^{\kappa}}\|A-BR\|_F.
  \)
  \section{Preliminaries}\label{sec2}

  \subsection{Restricted strong convexity and smoothness}\label{sec2.1}

  Restricted strong convexity (RSC) and restricted smoothness (RSS) are the common
  requirement for loss functions when handling low-rank matrix recovery problems
  (see, e.g., \cite{Negahban11,Negahban12,Li18,Zhu182,Zhang18}). Next we recall
  the concepts of RSC and RSS used in this paper.
  \begin{definition}\label{Def-RSC-RSS}
   A twice continuously differentiable function $\Psi\!:\mathbb{R}^{n\times m}\to\mathbb{R}$
   is said to satisfy the $(\kappa,\kappa)$-RSC of modulus $\alpha$ and the $(\kappa,\kappa)$-RSS
   of modulus $\beta$, respectively, if $0<\alpha\le\beta$ and for any
   $X,H\in\mathbb{R}^{n\times m}$ with ${\rm rank}(X)\le\kappa$ and ${\rm rank}(H)\le\kappa$,
   \begin{equation}\label{RSC-RSS}
    \alpha\|H\|^2_F\leq \nabla^2 \Psi(X)(H,H) \le\beta\|H\|^2_F.
   \end{equation}
  \end{definition}

  For the least squares loss in \eqref{Lsquare}, the $(\kappa,\kappa)$-RSC of modulus $\alpha$
  and $(\kappa,\kappa)$-RSS of modulus $\beta$ reduces to the $\kappa$-restricted
  smallest and largest eigenvalue of $\mathcal{A}^*\mathcal{A}$, i.e.,
  \[
    0<\alpha=\min_{{\rm rank}(X)\le\kappa,\|X\|_F=1}\|\mathcal{A}(X)\|^2
    \ \ {\rm and}\ \
    \beta=\max_{{\rm rank}(X)\le\kappa,\|X\|_F=1}\|\mathcal{A}(X)\|^2.
  \]
  Consequently, the $(\kappa,\kappa)$-RSC of modulus $\alpha=1-\delta_{\kappa}$ along with
  the $(\kappa,\kappa)$-RSS of modulus $\beta\!=1+\delta_{\kappa}$ for some $\delta_{\kappa}\in(0,1)$
  reduces to the RIP condition of the operator $\mathcal{A}$. From \cite{Recht10}, the least squares loss
  associated to many types of random sampling operators satisfy this property with a high probability.
  In addition, from the discussions in \cite{Zhu181,Li18}, some loss functions definitely
  have this property such as the weighted PCAs with positive weights,
  the noisy low-rank matrix recovery with noise matrix obeying Subbotin density
  \cite[Example 2.13]{Saumard14}, or the one-bit matrix completion with full observations.

  \medskip

  The following lemma improves a little the result of \cite[Proposition 2.1]{Li18}
  that requires $\Psi$ to have the $(2\kappa,4\kappa)$-RSC of modulus $\alpha$ and
  the $(2\kappa,4\kappa)$-RSS of modulus $\beta$.
  \begin{lemma}\label{RSC-RSS-ineq}
  Let $\Psi\!:\mathbb{R}^{n\times m}\to\mathbb{R}$ be a twice continuously
  differentiable function satisfying the $(\kappa,\kappa)$-RSC of modulus $\alpha$
  and the $(\kappa,\kappa)$-RSS of modulus $\beta$. Then,
  for any $X\in\!\mathbb{R}^{n\times m}$ with ${\rm rank}(X)\le\kappa$
  and any $Y,Z\in\!\mathbb{R}^{n\times m}$ with ${\rm rank}([Y\ \ Z])\le\kappa$,
  \[
    \Big|\frac{2}{\alpha+\beta}\nabla^2\Psi(X)(Y,Z)-\langle Y,Z\rangle\Big|
    \le\frac{\beta-\alpha}{\alpha+\beta}\big\|Y\big\|_F\big\|Z\big\|_F.
  \]
  \end{lemma}
  \begin{proof}
  Fix an arbitrary $(Y,Z)\in\mathbb{R}^{n\times m}\times\mathbb{R}^{n\times m}$
  with ${\rm rank}([Y\ Z])\le\kappa$. Fix any $X\in\mathbb{R}^{n\times m}$ with
  ${\rm rank}(X)\le\kappa$. If one of $Y$ and $Z$ is the zero matrix, the result is trivial.
  So, we assume that $Y\ne 0$ and $Z\ne 0$. Write $\overline{Y}:=\frac{Y}{\|Y\|_F}$
  and $\overline{Z}:=\frac{Z}{\|Z\|_F}$.
  Notice that ${\rm rank}([\overline{Y}\ \ \overline{Z}])\le\kappa$
  and ${\rm rank}(\overline{Y}\pm\overline{Z})\leq {\rm rank}([\overline{Y}\ \ \overline{Z}])$.
  Then, we have
  \begin{subequations}
  \begin{align*}
   \alpha\|\overline{Y}+\overline{Z}\|_F^2
   \le\nabla^2\Psi(X)\big(\overline{Y}+\overline{Z},\overline{Y}+\overline{Z}\big)
   \le\beta\big\|\overline{Y}+\overline{Z}\|_F^2,\\
   \alpha\|\overline{Y}-\overline{Z}\|_F^2
   \le\nabla^2\Psi(X)\big(\overline{Y}-\overline{Z},\overline{Y}-\overline{Z}\big)
   \le\beta\big\|\overline{Y}-\overline{Z}\|_F^2.
  \end{align*}
  \end{subequations}
  Along with $4|\nabla^2\Psi(X)\big(\overline{Y},\overline{Z}\big)|
  =\big|\nabla^2\Psi(X)(\overline{Y}\!+\!\overline{Z},\overline{Y}\!+\!\overline{Z})-
   \nabla^2\Psi(X)\big(\overline{Y}\!-\!\overline{Z},\overline{Y}\!-\!\overline{Z}\big)\big|$,
 \begin{subequations}
  \begin{align*}
  4\nabla^2\Psi(X)\big(\overline{Y},\overline{Z}\big)
  \le\beta\big\|\overline{Y}+\overline{Z}\|_F^2-\alpha\|\overline{Y}-\overline{Z}\|_F^2
  =2(\beta-\alpha)+2(\beta+\alpha)\big\langle \overline{Y},\overline{Z}\big\rangle,\\
  -4\nabla^2\Psi(X)\big(\overline{Y},\overline{Z}\big)
  \le\beta\big\|\overline{Y}-\overline{Z}\|_F^2-\alpha\big\|\overline{Y}+\overline{Z}\|_F^2
  =2(\beta-\alpha)-2(\beta+\alpha)\big\langle \overline{Y},\overline{Z}\big\rangle.\nonumber
 \end{align*}
 \end{subequations}
  The last two inequalities imply the desired inequality.
  The proof is then completed.
 \end{proof}

  From the reference \cite{Vershynin11}, we recall that a random variable $\xi$
  is called sub-Gaussian if
  \[
    K=\sup_{q\ge 1}q^{-1/2}(\mathbb{E}|\xi|^q)^{1/q}<\infty,
  \]
  and $K$ is referred to as the sub-Gaussian norm of $\xi$. Equivalently,
  the sub-Gaussian random variable $\xi$ satisfies the following bound for a constant $\tau^2$:
  \begin{equation}\label{tail}
    \mathbb{P}\big\{|\xi|>t\big\}\le 2e^{-t^2/(2\tau^2)}\ \ {\rm for\ all}\ t>0.
  \end{equation}
  We call the smallest $\tau^2$ satisfying \eqref{tail} the sub-Gaussian
  parameter. The tail-probability characterization in \eqref{tail} enables
  us to define centered sub-Gaussian random vectors.
  \begin{definition}\label{subgauss-bound}(see \cite{Cai10})
   A random vector $w=(w_1,\ldots,w_p)^{\mathbb{T}}$ is said to be a centered
   sub-Gaussian random vector if there exists $\tau>0$ such that for all $t>0$
   and all $\|v\|=1$,
   \[
     \mathbb{P}\big\{|v^{\mathbb{T}}w|>t\big\}\le 2e^{-t^2/(2\tau^2)}.
   \]
 \end{definition}

 \subsection{Properties of critical points of $\Phi_{\lambda}$}\label{sec2.2}

  To give the gradient and Hessian matrix of $\Phi_{\lambda}$,
  define $\Xi\!:\mathbb{R}^{n\times m}\to\mathbb{R}^{(n+m)\times (n+m)}$ by
  \begin{equation}\label{Ximap}
  \Xi(X):=\left(\begin{matrix}
                   \lambda I & \nabla\!f(X) \\
                   \nabla\!f(X)^\mathbb{T} & \lambda I\\
          \end{matrix}\right).
  \end{equation}
  For any given $(U,V)\in\mathbb{R}^{n\times r}\times\mathbb{R}^{m\times r}$,
  the gradient of $\Phi_{\lambda}$ at $(U,V)$ takes the form of
 \begin{equation}\label{gradient}
  \nabla\Phi_{\lambda}(U,V)=\left[\begin{matrix}
                           \nabla\!f(X)V+\lambda{U} \\
                           \nabla\!f(X)^\mathbb{T}{U}+\lambda{V}
                          \end{matrix}\right]
               =\Xi(X)W\ \ {\rm with}\ \ X=UV^{\mathbb{T}};
  \end{equation}
  and for any $\Delta=(\Delta_U;\Delta_V)$ with $\Delta_U\in\mathbb{R}^{n\times r}$
  and $\Delta_V\in\mathbb{R}^{m\times r}$, it holds that
  \begin{align}\label{Hessian}
   \nabla^2\Phi_{\lambda}(U,V)(\Delta,\Delta)
   &=\nabla^2\!f(X)(U\Delta_V^\mathbb{T}+\Delta_U V^\mathbb{T},U\Delta_V^\mathbb{T}+\Delta_U V^\mathbb{T})\nonumber\\
    &\quad+2\langle\nabla\!f(X),\Delta_U \Delta_V^\mathbb{T} \rangle+\lambda\langle \Delta,\Delta\rangle.
  \end{align}
  By invoking \eqref{gradient}, it is easy to obtain the balance property
  of the critical points of $\Phi_{\lambda}$.
 \begin{lemma}\label{stationary-lemma1}
  Fix any $\lambda>0$. Every critical point of $\Phi_{\lambda}$ belongs to the following set
 \begin{equation}\label{ME-lam}
   \mathcal{E}_{\lambda}:=\Big\{(U,V)\in\mathbb{R}^{n\times r}\!\times\mathbb{R}^{m\times r}
   \ |\ U^\mathbb{T}U=V^\mathbb{T}V\Big\},
 \end{equation}
 and consequently, any critical point $(U,V)$ of $\Phi_{\lambda}$ satisfies
 \[
   {\rm rank}(UV^{\mathbb{T}})={\rm rank}(U)={\rm rank}(V)={\rm rank}(W).
 \]
 \end{lemma}
 \begin{proof}
  The first part is given in \cite[Proposition 4.3]{Li18}. So, it suffices to prove the second part.
  From the SVD of $U$ and $V$, we have $[\sigma(U)]^2=\sigma(U^\mathbb{T}U)$
  and $[\sigma(V)]^2=\sigma(V^\mathbb{T}V)$. Along with $U^\mathbb{T}U=V^\mathbb{T}V$,
  $\sigma(U)=\sigma(V)$ and ${\rm rank}(U)\!=\!{\rm rank}(V)$.
  Let $X=UV^\mathbb{T}$. From $U^\mathbb{T}U=V^\mathbb{T}V$, we have
 $X^\mathbb{T}X=VU^\mathbb{T}UV^\mathbb{T}=VV^\mathbb{T}VV^\mathbb{T}$
 and $W^\mathbb{T}W=2V^\mathbb{T}V$, which means that $\sigma(X)=[\sigma(V)]^2$
 and $\sigma(W)=\sqrt{2}\sigma(V)$. Then, ${\rm rank}(X)={\rm rank}(V)={\rm rank}(W)$.
 \end{proof}

 When $f$ has a special structure, the critical points of $\Phi_{\lambda}$
 have a favorable property.
 \begin{lemma}\label{stationary-lemma2}
  Let $f(X):=\frac{1}{2}\|X-D\|_F^2$ for $X\in\mathbb{R}^{n\times m}$,
  where $D={\rm Diag}(d_1,\ldots,d_{n})$ is an $n\times m$ rectangular matrix
  for $d_1\ge\cdots\ge d_{n}\ge 0$ with $d_{r^*}>0$. Then,
 \begin{itemize}
  \item [(i)] for any critical point $(U,V)$ of $\Phi_{\lambda}$ associated to $\lambda>0$,
              it holds that $U_1=V_1$;

  \item [(ii)] the critical point set of $\Phi_{\lambda}$ associated to $\lambda>d_{r^*+1}$
               takes the following form
               \begin{equation}\label{critical-equa}
               \!{\rm crit}\Phi_{\lambda}
               =\left\{(U,V)\in\mathbb{R}^{n\times r}\times\mathbb{R}^{m\times r}\ \bigg|\left.\begin{array}{c}
                      {U}_1={V}_1,\ {U}_2=0,\ {V}_2=0\\
                      ({U}_1{U}_1^\mathbb{T}-D_1+\lambda I){U}_1=0
                      \end{array}\right.\!\right\}
               \end{equation}
             where $D_1={\rm Diag}(d_1,\ldots,d_{r^*})\in\mathbb{R}^{r^*\times r^*}$,
             $D_2={\rm Diag}(d_{r^*+1},\ldots,d_{n})\in\mathbb{R}^{(n-r^*)\times (m-r^*)}$.
 \end{itemize}
 \end{lemma}
 \begin{proof}
 {\bf(i)} Pick a critical point $(U,V)$ of $\Phi_{\lambda}$ associated to $\lambda>0$.
 From \eqref{gradient}, we have
 \begin{subnumcases}{}
  \label{Phi-gradient-U-1}
  0=\nabla_1\Phi_{\lambda}(U,V)
    =\left[\begin{matrix}
     U_{1}(V^\mathbb{T}V+\lambda I)- D_1V_1 \\
     U_{2}(V^\mathbb{T}V+\lambda I)-D_2V_2
     \end{matrix}\right],\\
 \label{Phi-gradient-V-1}
  0=\nabla_2\Phi_{\lambda}(U,V)
    =\left[\begin{matrix}
     V_{1}(U^\mathbb{T}U+\lambda I)-D_1U_1 \\
     V_{2}(U^\mathbb{T}U+\lambda I)-D_2^{\mathbb{T}}U_2
     \end{matrix}\right]
  \end{subnumcases}
 where $\nabla_1\Phi_{\lambda}(U',V')$ and $\nabla_2\Phi_{\lambda}(U',V')$
 are the partial gradient of $\Phi_{\lambda}$ w.r.t. variable $U$
 and variable $V$, respectively, at $(U',V')$.
 By combining \eqref{Phi-gradient-U-1}-\eqref{Phi-gradient-V-1} with
 Lemma \ref{stationary-lemma1},
 \[
  (U_{1}-V_{\!1})(V^\mathbb{T}V+\lambda I)
  +D_1(U_1-V_{\!1})=0.
 \]
 Let $V^\mathbb{T}V+\lambda I$ have the spectral decomposition
 as $P\Lambda P^{\mathbb{T}}$ with $\Lambda={\rm Diag}(\mu_1,\ldots,\mu_{r})$ and
 $P\in\mathbb{O}^{r}$. Clearly, $\mu_i>0$ for $i=1,\ldots,r$.
 Then, the last equality can be rewritten as
 \[
  (U_{1}-V_{\!1})P\Lambda +D_1(U_1-V_{\!1})P=0.
 \]
 Since $\Lambda$ and $D_1$ are diagonal and positive definite, we have
 $(U_{1}-V_{\!1})P=0$ and $U_1=V_1$.

 \noindent
 {\bf(ii)} Pick any $(U,V)$ of ${\rm crit}\Phi_{\lambda}$
 with $\lambda>d_{r^*+1}$. By the second equality in \eqref{Phi-gradient-U-1} and \eqref{Phi-gradient-V-1},
 \begin{align*}
  0&=\|U_{2}(V^\mathbb{T}V\!+\lambda I)\!-D_2V_2\|_F
      +\|V_{2}(U^\mathbb{T}U+\lambda I)-D_2^{\mathbb{T}}U_2\|_F\\
  &\ge \lambda\|U_{2}\|_F-d_{r^*+1}\|V_2\|_F+\lambda\|V_2\|_F-d_{r^*+1}\|U_2\|_F\\
  &=(\lambda-d_{r^*+1})(\|U_{2}\|_F+\|V_2\|_F).
 \end{align*}
 Since $\lambda>d_{r^*+1}$, this implies that $U_{2}=0$ and $V_{2}=0$.
 Together with part (i), we conclude that $(U,V)$ belongs to the set
 on the right hand side of \eqref{critical-equa}. Conversely, for any $(U,V)$ from
 the set on the right hand side of \eqref{critical-equa}, it is clear that
 $\nabla_1\Phi_{\lambda}(U,V)=\nabla_2\Phi_{\lambda}(U,V)=0$, i.e.,
 $(U,V)\in{\rm crit}\Phi_{\lambda}$. Thus, we complete the proof.
 \end{proof}
 \subsection{KL property of an lsc function}\label{sec2.3}

 \begin{definition}\label{KL-Def1}
  Let $h\!:\mathbb{R}^{n\times m}\!\to(-\infty,\infty]$ be a proper function.
  The function $h$ is said to have the Kurdyka-{\L}ojasiewicz (KL) property
  at $\overline{x}\in{\rm dom}\,\partial h$ if there exist $\eta\in(0,\infty]$,
  a continuous concave function $\varphi\!:[0,\eta)\to\mathbb{R}_{+}$ satisfying
  \begin{itemize}
    \item [(i)] $\varphi(0)=0$ and $\varphi$ is continuously differentiable on $(0,\eta)$;

    \item[(ii)] for all $s\in(0,\eta)$, $\varphi'(s)>0$,
  \end{itemize}
  and a neighborhood $\mathcal{U}$ of $\overline{x}$ such that for all
  \(
    x\in\mathcal{U}\cap\big[h(\overline{x})<h<h(\overline{x})+\eta\big],
  \)
  \[
    \varphi'(h(x)-h(\overline{x})){\rm dist}(0,\partial h(x))\ge 1.
  \]
  If $\varphi$ can be chosen as $\varphi(s)=c\sqrt{s}$ for some $c>0$,
  then $h$ is said to have the KL property with an exponent of $1/2$ at $\overline{x}$.
  If $h$ has the KL property of exponent $1/2$ at each point of ${\rm dom}\,\partial h$,
  then $h$ is called a KL function of exponent $1/2$.
 \end{definition}
 \begin{remark}\label{KL-remark}
  To show that a proper function is a KL function of exponent $1/2$,
  it suffices to verify if it has the KL property of exponent $1/2$
  at all critical points since, by \cite[Lemma 2.1]{Attouch10},
  it has this property at all noncritical points.
 \end{remark}
 \section{Error bound of critical points}\label{sec3}

  The following lemma states an important property for any critial point $(U,V)$
  of $\Phi_{\lambda}$ with ${\rm rank}(U)\le r^*$ or ${\rm rank}(V)\le r^*$,
  whose proof is included in Appendix A.
 \begin{lemma}\label{fcond-asym}
  Suppose that $f$ has the $(2r^*,4r^*)$-RSC of modulus $\alpha$ and
  the $(2r^*,4r^*)$-RSS of modulus $\beta$. Fix any $\lambda>0$
  and any $(U^*,V^*)\in\mathcal{E}^*$. Then, for any critical point $(U,V)$
  of $\Phi_{\lambda}$ with ${\rm rank}(UV^{\mathbb{T}})\leq r^*$ and
  any column orthonormal $Q$ spanning ${\rm col}(W)$,
  \begin{align*}
  & \frac{1}{2}\|(WW^\mathbb{T}-W^*{W^*}^\mathbb{T})QQ^\mathbb{T} \|^2_F
      +\frac{2}{\alpha+\beta}\langle\Xi(M^*),(WW^\mathbb{T}-W^*{W^*}^\mathbb{T})QQ^\mathbb{T}\rangle\nonumber\\
  &\qquad\qquad\qquad\quad
  \le\frac{\beta-\alpha}{\alpha+\beta}\|UV^\mathbb{T}-U^*{V^*}^\mathbb{T}\|_F
  \|(WW^\mathbb{T}-W^*{W^*}^\mathbb{T})QQ^\mathbb{T} \|_F.
 \end{align*}
 \end{lemma}
 \begin{remark}\label{remark31}
  Write $\Gamma\!:=(WW^\mathbb{T}\!-\!W^*{W^*}^\mathbb{T})QQ^\mathbb{T}$.
  The result of Lemma \ref{fcond-asym} implies that
 \begin{equation}\label{ineq-remark31}
  \frac{1}{2}\|\Gamma\|^2_F
  \le\frac{2\sqrt{r^*}}{\alpha+\beta}\|\Xi(M^*)\|\|\Gamma\|_F
   +\frac{\beta-\alpha}{\alpha+\beta}\|UV^\mathbb{T}\!-\!U^*{V^*}^\mathbb{T}\|_F
  \|\Gamma\|_F.
 \end{equation}
 Recall that $2|ab|\le\gamma a^2+\gamma^{-1} b^2$ for any $a,b\in\mathbb{R}$ and any $\gamma>0$.
 Then, it holds that
 \begin{align*}
 \frac{2\sqrt{r^*}}{\alpha+\beta}\|\Xi(M^*)\|\|\Gamma\|_F
 \le\frac{64r^*}{(\alpha+\beta)^2}\|\Xi(M^*)\|^2+\frac{1}{64}\|\Gamma\|^2_F,\qquad\qquad\\
 \frac{\beta-\alpha}{\alpha+\beta}\|UV^\mathbb{T}\!-\!U^*{V^*}^\mathbb{T}\|_F\|\Gamma\|_F
  \le\Big(\frac{\beta-\alpha}{\alpha+\beta}\Big)^2\|UV^\mathbb{T}\!-\!U^*{V^*}^\mathbb{T}\|^2_F
  +\frac{1}{4}\|\Gamma\|^2_F.
 \end{align*}
 Together with \eqref{ineq-remark31} and
 $\sqrt{2}\|UV^{\mathbb{T}}-U^*{V^*}^{\mathbb{T}}\|_F\le\|WW^{\mathbb{T}}-W^*{W^*}^{\mathbb{T}}\|_F$,
 it follows that
 \begin{equation}\label{ineq-fcond}
   \frac{15}{64}\|\Gamma\|^2_F
   \le\frac{64r^*}{(\alpha\!+\!\beta)^2}\|\Xi(M^*)\|^2
    +\frac{(\beta\!-\!\alpha)^2}{2(\alpha\!+\!\beta)^2}\|WW^\mathbb{T}\!-\!W^*{W^*}^\mathbb{T}\|_F^2.
 \end{equation}
 So, $\|WW^\mathbb{T}\!-\!W^*{W^*}^\mathbb{T}\|_F^2$ is lower bounded
 by $\max\big(0,\frac{15(\alpha+\beta)^2}{32(\beta-\alpha)^2}\|\Gamma\|^2_F
   -\frac{128r^*}{(\beta-\alpha)^2}\|\Xi(M^*)\|^2\big)$.
 \end{remark}

 When the critical point $(U,V)$ in Lemma \ref{fcond-asym} satisfies
 $\nabla^2\Phi_{\lambda}(U,V)(\Delta,\Delta)\ge 0$ for
 \begin{equation}\label{dir-Delta}
  \Delta:=W-[W^*\ \ 0]R^*\quad{\rm with}\ \ R^*\in\mathop{\arg\min}_{R'\in\mathbb{O}^r}\|W-[W^*\ \ 0]R'\|_F,
 \end{equation}
 the following lemma gives a lower bound for $\|W\Delta^{\mathbb{T}}\|_F^2$;
 see Appendix B for its proof.
 \begin{lemma}\label{scond-asym}
  Suppose that $f$ has the $(2r^*,4r^*)$-RSC of modulus $\alpha$ and
  the $(2r^*,4r^*)$-RSS of modulus $\beta$. Fix any $\lambda>0$
  and any $(U^*,V^*)\in\mathcal{E}^*$. Consider a critical point $(U,V)$
  of $\Phi_{\lambda}$ with ${\rm rank}(UV^{\mathbb{T}})\le r^*$. Then,
  for the direction $\Delta=(\Delta_U;\Delta_V)$ defined in \eqref{dir-Delta},
 \begin{align}\label{goal-equa}
  \nabla^2\Phi_{\lambda}(U,V)(\Delta,\Delta)
  &=\nabla^2\!f(UV^{\mathbb{T}})(U\Delta_V^\mathbb{T}\!+\!\Delta_U V^\mathbb{T},U\Delta_V^\mathbb{T}\!+\!\Delta_U V^\mathbb{T}) \nonumber\\
  &\quad-\langle\Xi(UV^{\mathbb{T}}),WW^\mathbb{T}-W^*(W^*)^\mathbb{T}\rangle.
 \end{align}
 If in addition $\nabla^2\Phi_{\lambda}(U,V)(\Delta,\Delta)\ge 0$,
  then it holds that
 \begin{equation}\label{goal-ineq}
   \|W\Delta^\mathbb{T}\|_F^2
  \ge\max\Big(0,\frac{\alpha}{2\beta} \|WW^\mathbb{T}\!-W^*(W^*)^\mathbb{T}\|^2_F+
  \frac{1}{\beta}\langle \Xi(M^*),WW^\mathbb{T}\!-\!W^*(W^*)^\mathbb{T}\rangle\Big).
 \end{equation}
 \end{lemma}

 Now we state one of the main results which provides an upper bound to the true $M^*$
 for those non-strict critical points of $\Phi_{\lambda}$ with rank at most $r^*$.
 \begin{theorem}\label{error-theorem}
  Suppose that $f$ satisfies the $(2r^*,4r^*)$-RSC of modulus $\alpha$
  and the $(2r^*,4r^*)$-RSS of modulus $\beta$, respectively,
  with $\beta/\alpha\le 1.38$. Fix any $\lambda>0$. Then,
  for any critical point $(U,V)$ of $\Phi_{\lambda}$ with ${\rm rank}(UV^{\mathbb{T}})\le r^*$
  and $\nabla^2\Phi_{\lambda}(U,V)$ being PSD,
  there exists a constant $\widehat{\gamma}>0$ (depending only on $\alpha$ and $\beta$)
  such that the following inequality holds
 \begin{equation}\label{local-true-bound}
  2\|UV^{\mathbb{T}}\!-\!M^*\|_F^2\le\|WW^\mathbb{T}\!-\!W^*{W^*}^\mathbb{T}\|^2_F
  \le \widehat{\gamma}\|\Xi(M^*)\|^2\le 2\widehat{\gamma} r^*\big[\lambda^2+\!\|\nabla\!f(M^*)\|^2\big].
 \end{equation}
 \end{theorem}
 \begin{proof}
  Pick any $(U^*,V^*)\in\mathcal{E}^*$. Let $R^*$ be given by \eqref{dir-Delta} with $W$ and $W^*$,
  and let $R_1^*$ be the matrix consisting of the first $r^*$ rows of $R^*$.
  Note that $R^*\in{\displaystyle\mathop{\arg\max}_{R'\in\mathbb{O}^r}}\langle[W^*\ \ 0]^{\mathbb{T}}W,R'\rangle$.
  One can check that $W^{\mathbb{T}}W^*R_1^*=(W^*R_1^*)^{\mathbb{T}}W$ is PSD.
  By \cite[Lemma 3.6]{Li18},
  \begin{equation}\label{WDeta}
  \|W\Delta^\mathbb{T}\|_F^2\leq\frac{1}{8} \|WW^\mathbb{T}\!-\!W^*{W^*}^\mathbb{T}\|_F^2
  +\Big(3+\!\frac{1}{2\sqrt{2}-2}\Big)\big\|(WW^\mathbb{T}\!-\!W^*{W^*}^\mathbb{T})QQ^\mathbb{T}\big\|_F^2.
 \end{equation}
 Since $\nabla^2\Phi_{\lambda}(U,V)$ is PSD, by invoking inequality \eqref{goal-ineq}
 in Lemma \ref{scond-asym}, it follows that
 \begin{align*}
  \|W\Delta^\mathbb{T}\|_F^2
  &\ge\frac{\alpha}{2\beta} \|WW^\mathbb{T}\!-W^*{W^*}^\mathbb{T}\|^2_F
      +\frac{1}{\beta}\langle \Xi(M^*),WW^\mathbb{T}-W^*{W^*}^\mathbb{T}\rangle\\
  &\ge\frac{\alpha}{2\beta} \|WW^\mathbb{T}-W^*{W^*}^\mathbb{T}\|_F^2-\frac{\sqrt{2r^*}}{\beta}\|\Xi(M^*)\|
  \|WW^\mathbb{T}-W^*{W^*}^\mathbb{T}\|_F\\
   &\ge\frac{\alpha}{2\beta} \|WW^\mathbb{T}-W^*{W^*}^\mathbb{T}\|_F^2-\frac{64r^*}{\alpha\beta}\|\Xi(M^*)\|^2
  -\frac{\alpha}{128\beta}\|WW^\mathbb{T}-W^*{W^*}^\mathbb{T}\|^2_F\\
  &=\frac{63\alpha}{128\beta}\|WW^\mathbb{T}-W^*{W^*}^\mathbb{T}\|^2-\frac{64r^*}{\alpha\beta}\|\Xi(M^*)\|^2
 \end{align*}
  where the third inequality is due to $2|ab|\le\gamma a^2+\gamma^{-1} b^2$ for any $a,b\in\mathbb{R}$ and any $\gamma>0$.
  From the last two inequalities, it is not hard to obtain that
 \begin{align}\label{temp-ineq31}
 \frac{63\alpha}{128\beta} \|WW^\mathbb{T}-W^*{W^*}^\mathbb{T}\|_F^2
 &\le \frac{1}{8}\big\|WW^\mathbb{T}-W^*{W^*}^\mathbb{T}\big\|_F^2+\frac{64r^*}{\alpha\beta}\|\Xi(M^*)\|^2\nonumber\\
 &\quad+\frac{7\!+\!\sqrt{2}}{2}\big\|(WW^\mathbb{T}\!-\!W^*{W^*}^\mathbb{T})QQ^\mathbb{T}\big\|_F^2.
 \end{align}
 Combining this inequality with inequality \eqref{ineq-fcond} yields that
 \begin{align*}
  &\Big(\frac{63\alpha}{128\beta}-\frac{1}{8}-\frac{16(7\!+\!\sqrt{2})(\beta-\alpha)^2}{15(\alpha+\beta)^2}\Big)
    \|WW^\mathbb{T}\!-\!W^*{W^*}^\mathbb{T}\|^2_F\\
 &\qquad\qquad\qquad\qquad
 \le\frac{2048(7\!+\!\sqrt{2})r^*}{15(\alpha+\beta)^2}\|\Xi(M^*)\|^2+\frac{64r^*}{\alpha\beta}\|\Xi(M^*)\|^2.
 \end{align*}
 Since $\beta/\alpha\le 1.38$, we have
 \(
   \gamma_1=\frac{63\alpha}{128\beta}-\frac{1}{8}-\frac{16(7+\sqrt{2})(\beta-\alpha)^2}{15(\alpha+\beta)^2}>0.
 \)
 So, the desired inequality holds with $\widehat{\gamma}=\gamma_2/\gamma_1$ for
 $\gamma_2=\frac{2048(7+\sqrt{2})}{15(\alpha+\beta)^2}+\frac{64}{\alpha\beta}$
 and $\|\Xi(M^*)\|^2\leq 2(\lambda^2\!+\!\|\nabla\!f(M^*)\|^2)$.
 \end{proof}
 \begin{remark}\label{main-remark1}
  {\bf(i)} For the least squares loss in \eqref{Lsquare}, the term $\|\nabla\!f(M^*)\|$
  in the error bound reduces to $\|\mathcal{A}^*(\omega)\|$. In this case,
  if $\mathcal{A}$ is a noiseless and full sampling operator, every local minimizer
  $(\overline{U},\overline{V})$ of problem \eqref{Phi-lam} with ${\rm rank}(\overline{U})\le r^*$
  or ${\rm rank}(\overline{V})\le r^*$ has the error bound
  $\|\overline{U}\overline{V}^\mathbb{T}\!-\!M^*\|_F\le \sqrt{\widehat{\gamma}\,r^*}\lambda$.
  By the characterization in \cite{ZhangSo18} for the global optimal solution set of \eqref{Phi-lam}
  with $r=r^*$, each global optimal solution $(\overline{U}^*,\overline{V}^*)$ satisfies
   $\|\overline{U}^*(\overline{V}^*)^{\mathbb{T}}\!-\!M^*\|_F\le\sqrt{r^*}\lambda$.
  This shows that the error bound in Theorem \ref{error-theorem} is optimal.

  \medskip
  \noindent
  {\bf(ii)} From \cite{Lee19} many kinds of first-order methods for problem \eqref{Phi-lam}
  can find a critical point $(U,V)$ with a PSD Hessian $\nabla^2\Phi_{\lambda}(U,V)$ in a high probability.
  Along with Theorem \ref{error-theorem}, when solving \eqref{Phi-lam} with one of
  these first-order methods, if the obtained critical point has rank at most $r^*$,
  it is highly possible for it to have a desirable error bound to the true $M^*$.

  \medskip
  \noindent
  {\bf(iii)} By combining Theorem \ref{result-MS} with Lemma \ref{lemma2} in Appendix C,
  it follows that each optimal solution $\overline{X}_{\rm cr}$ of the convex problem
  \eqref{Nuclear-norm} with ${\rm rank}(\overline{X}_{\rm cr})\!\le r^*$ satisfies
  \[
    \|\overline{X}_{\rm cr}-M^*\|_F\le\sqrt{\widehat{\gamma}r^*}\big[\lambda+\|\nabla\!f(M^*)\|\big],
  \]
  which is consistent with the one in \cite[Corollary 1]{Negahban11} for the optimal
  solution (though it is unknown whether its rank is less than $r^*$ or not) of
  the convex relaxation approach. This implies that the error bound
  of the convex relaxation approach is near optimal.
 \end{remark}

  Next we illustrate the result of Theorem \ref{error-theorem} via two
  specific observation models.
  \subsection{Matrix sensing}\label{sec3.1}

  The matrix sensing problem aims to recover the true matrix $M^*$ via the observation
  model \eqref{observe-model}, where the sampling operator $\mathcal{A}$ is defined
  by $[\mathcal{A}(Z)]_i\!:=\langle A_i,Z\rangle$ for $i=1,\ldots,p$,
  and the entries $\omega_1,\ldots,\omega_p$ of the noise vector $\omega$ are
  assumed to be i.i.d. sub-Gaussian of parameter $\sigma_{\omega}^2$.
  By Definition \ref{subgauss-bound} and the discussions in \cite[Page 24]{Datta17},
  for every $u\in \mathbb{R}^p$, there exists an absolute constant $\widehat{c}>0$
  such that with probability at least $1-\frac{1}{nm}$,
  \begin{equation}\label{noise-ineq}
   \big|{\textstyle\sum_{i=1}^p} u_i\omega_i \big|
   \le \widehat{c}\sigma_\omega\sqrt{\ln(nm)}\|u\|.
  \end{equation}

  \vspace{-0.3cm}
 \begin{assumption}\label{RIP}
  The sampling operator $\mathcal{A}$ has the $4r^*$-RIP of constant
  $\delta_{4r^*}\in(0,\frac{19}{119})$.
 \end{assumption}

  Take $f(Z):=\frac{1}{2p}\|\mathcal{A}(Z)-y\|^2$ for $Z\in\mathbb{R}^{n\times m}$.
  Then, under Assumption \ref{RIP}, the loss function $f$ satisfies the conditions
  of Theorem \ref{error-theorem} with $\beta=\frac{1+\delta_{4r^*}}{p}$ and $\alpha=\frac{1-\delta_{4r^*}}{p}$.
  We next upper bound $\|\mathcal{A}^*(\omega)\|$.
  Let $\mathcal{S}^{n-1}=\{u\in\mathbb{R}^{n}\ |\ \|u\|=1\}$ denote
  the Euclidean sphere in $\mathbb{R}^{n}$. From the variational
  characterization of the spectral norm of matrices,
  \[
    \|\mathcal{A}^*(\omega)\|
    =\sup_{u\in\mathcal{S}^{n-1},v\in\mathcal{S}^{m-1}}\langle u,\mathcal{A}^*(\omega)v\rangle
    =\sup_{u\in\mathcal{S}^{n-1},v\in\mathcal{S}^{m-1}}{\textstyle\sum_{i=1}^p}\omega_i\langle A_i,uv^{\mathbb{T}}\rangle.
  \]
  By invoking \eqref{noise-ineq} and the RIP of $\mathcal{A}$,
  with probability at least $1-\frac{1}{nm}$ it holds that
 \begin{align*}
   \|\mathcal{A}^*(\omega)\|
   &\le\widehat{c}\sigma_{\omega}\sqrt{\ln(nm)}\sup_{u\in\mathcal{S}^{n-1},v\in\mathcal{S}^{m-1}}
    \|\mathcal{A}(uv^{\mathbb{T}})\|\\
   &\le\widehat{c}\sigma_{\omega}\sqrt{\ln(nm)}\sqrt{1+\delta_{4r^*}}
   \sup_{u\in\mathcal{S}^{n-1},v\in\mathcal{S}^{m-1}}\|uv^{\mathbb{T}}\|_F\\
   &\le \widehat{c}\sigma_{\omega}\sqrt{\ln(nm)}\sqrt{1+\delta_{4r^*}}.
  \end{align*}
  Notice that $\nabla\!f(M^*)=\frac{1}{p}\mathcal{A}^*(\omega)$.
  By Theorem \ref{error-theorem}, we obtain the following conclusion.
 \begin{corollary}\label{result-MS}
  Suppose $f(\cdot)=\frac{1}{2p}\|\mathcal{A}(\cdot)-y\|^2$ with
  $\mathcal{A}$ satisfying Assumption \ref{RIP}. Then, for any critical
  point $(U,V)$ of \eqref{Phi-lam} with ${\rm rank}(UV^{\mathbb{T}})\le r^*$
  and $\nabla^2\Phi_{\lambda}(U,V)$ being PSD,
  \begin{equation}\label{ineq-Msensing}
  \sqrt{2}\|UV^{\mathbb{T}}-M^*\|_F\le\|WW^\mathbb{T}\!-\!W^*{W^*}^\mathbb{T}\|_F
   \le \phi(\delta_{4r^*})\sqrt{r^*}\big(\lambda+\!\sigma_{\omega}\sqrt{\ln(n m)}/p\big),
 \end{equation}
 holds w.p. at least $1-\frac{1}{nm}$, where $\phi(\delta_{4r^*})$
 is a nondecreasing positive function of $\delta_{4r^*}$.
 When $\lambda=c\sigma_{\omega}\!\frac{\sqrt{\ln(nm)}}{p}$
 for an absolute constant $c>0$, w.p. at least $1-\frac{1}{nm}$ we have
 \[
  \|UV^{\mathbb{T}}-M^*\|_F\asymp O(\sigma_{\omega}\sqrt{r^*\ln(nm)}/p).
 \]
 \end{corollary}

 \subsection{Weighted principle component analysis}\label{3.3}

 The weighted PCA problem aims to recover an unknown true matrix $M^*\in\mathbb{R}^{n\times m}$
 from an elementwise weighted observation $Y=H\circ(M^*+E)$,
 where $H$ is the positive weight matrix, $E$ is the noise matrix,
 and ``$\circ$'' denotes the Hadamard product of matrices. This corresponds to
 the observation model \eqref{observe-model} with $\mathcal{A}(Z):={\rm vec}(H\circ Z)$
 and $\omega=\mathcal{A}(E)$. We assume that the entries $E_{ij}$ of $E$ are
 i.i.d. sub-Gaussian random variables of parameter $\sigma_{E}^2$.
 By Definition \ref{subgauss-bound} and the discussions in \cite[Page 24]{Datta17},
 for every $H\in\mathbb{R}^{n\times m}$, there exists an absolute constant
 $\widetilde{c}>0$ such that with probability at least $1-\frac{1}{nm}$,
 \begin{equation}\label{noise-PCA}
  \big|{\textstyle\sum_{i,j}}H_{ij}E_{ij}\big|
  \le\widetilde{c}\sigma_E\sqrt{\ln(nm)}\|H\|_F.
 \end{equation}
 Take $f(X):=\frac{1}{2}\|H\circ[X-(M^*+E)]\|_F^2$ for $X\in\mathbb{R}^{n\times m}$.
 Then, for each $X,\Delta\in\mathbb{R}^{n\times m}$,
 \[
   \nabla\!f(X)=H\circ H\circ[X-(M^*\!+\!E)]\ \ {\rm and}\ \
   \nabla^2f(X)[\Delta,\Delta]=\|H\circ\Delta\|_F^2.
 \]
 Clearly, $f$ satisfies the $(2r^*,4r^*)$-RSC of modulus $\alpha=\|H\|_{\min}^2$
 and $(2r^*,4r^*)$-RSS of modulus $\beta=\|H\|_{\max}^2$, where $\|H\|_{\min}\!:=\min_{i,j}H_{ij}$
 and $\|H\|_{\max}\!:=\max_{i,j}H_{ij}$. Note that
 \begin{align*}
  \|\nabla\!f(M^*)\|=\|H\circ H\circ E\|
  &=\sup_{u\in\mathcal{S}^{n-1},v\in\mathcal{S}^{m-1}}\langle u,(H\circ H\circ E)v\rangle\\
  &=\sup_{u\in\mathcal{S}^{n-1},v\in\mathcal{S}^{m-1}}\sum_{i,j}u_iH_{ij}^2E_{ij}v_j.
 \end{align*}
 By invoking \eqref{noise-PCA} and the $(2r^*,4r^*)$-RSS of $f$,
 with probability at least $1-\frac{1}{nm}$ we have
 \[
  \|\nabla\!f(M^*)\|\le \widetilde{c}\sigma_E\sqrt{\ln(nm)}\sup_{u\in\mathcal{S}^{n-1},v\in\mathcal{S}^{m-1}}\|H\circ H\circ(uv^{\mathbb{T}})\|_F\le \widetilde{c}\sigma_E\sqrt{\ln(nm)}\|H\|_{\max}^2.
 \]
 By invoking Theorem \ref{error-theorem} with this loss function,
 we have the following conclusion.
 \begin{corollary}\label{result-WPCA}
  Suppose that $f(\cdot)=\frac{1}{2}\|H\circ[\,\cdot-(M^*\!+\!E)]\|_F^2$.
  Then, for any critical point $(U,V)$ of problem \eqref{Phi-lam} with
  ${\rm rank}(UV^{\mathbb{T}})\le r^*$ and $\nabla^2\Phi_{\lambda}(U,V)$ being PSD,
  \begin{equation*}
  \sqrt{2}\|UV^{\mathbb{T}}-M^*\|_F\!\le\!\|WW^\mathbb{T}\!-\!W^*{W^*}^\mathbb{T}\|_F
 \! \le\! \phi\big(\!\|H\|_{\rm max},\!\|H\|_{\rm min}\!\big)
  \sqrt{r^*}\big(\lambda\!+\!\sigma_{\!E}\sqrt{\ln(nm)}\|H\|_{\max}^2\big)
 \end{equation*}
 holds with probability at least $1-\frac{1}{nm}$,
 where $\phi(\|H\|_{\rm max},\|H\|_{\rm min}):=\widetilde{\phi}\big(\frac{\|H\|_{\rm max}}{\|H\|_{\rm min}}\big)/\|H\|^{2}_{\rm min}$ and $\widetilde{\phi}\big(\frac{\|H\|_{\rm max}}{\|H\|_{\rm min}}\big)$
 is a nondecreasing positive function of $\frac{\|H\|_{\rm max}}{\|H\|_{\rm min}}$.
 For $\lambda=c\sigma_E\sqrt{\ln(nm)}\|H\|_{\max}^2$
 with an absolute constant $c>0$, it holds with probability at least $1-\frac{1}{nm}$ that
 \[
    \|UV^{\mathbb{T}}-M^*\|_F\asymp O(\sigma_{\!E}\sqrt{r^*\ln(nm)}\|H\|_{\max}^2).
 \]
 \end{corollary}
 \section{KL property of exponent $1/2$}\label{sec4}

 In this section, we focus on the KL property of exponent $1/2$ of $\Phi_{\lambda}$
 with $r=r^*$ under the noisy full and partial sample setting, respectively.
 Unless otherwise stated, $r=r^*$.
 \subsection{Noisy and full sample setting}\label{sec4.1}

 Now $f(X):=\frac{1}{2}\|X-M\|^2$ for $X\!\in\mathbb{R}^{n\times m}$,
 where $M\!=M^*+E$ is a noisy observation on the true $M^*$. Write
 $\Sigma:={\rm Diag}(\sigma(M))\in\mathbb{R}^{n\times m},
 \Sigma_1:=\!{\rm Diag}(\sigma_1(M),\ldots,\sigma_{r}(M))\in\mathbb{R}^{r\times r}$
 and $\Sigma_2\!:={\rm Diag}(\sigma_{r+1}(M),\ldots,\sigma_{n}(M))\!\in\!\mathbb{R}^{(n-r)\times(m-r)}$.
 Then problem \eqref{Phi-lam} is equivalent to
 \begin{equation}\label{Phi-lam-1}
  \min_{U\in \mathbb{R}^{n\times r},V\in \mathbb{R}^{m\times r}}
  \Big\{\widetilde{\Phi}_{\lambda}(U,V):=\frac{1}{2}\|UV^\mathbb{T}-\Sigma\|_F^2+\frac{\lambda}{2}\big(\|U\|_F^2+\|V\|_F^2\big)\Big\}
 \end{equation}
 in the sense that if $(\overline{U}^*,\overline{V}^*)$ is a global optimal solution of \eqref{Phi-lam-1},
 then $(P\overline{U}^*,Q\overline{V}^*)$ with $(P,Q)\in\mathbb{O}^{n,m}(M)$
 is globally optimal to problem \eqref{Phi-lam}; and conversely, if $(\overline{U}^*,\overline{V}^*)$
 is a global optimal solution of \eqref{Phi-lam}, then $(P^{\mathbb{T}}\overline{U}^*,Q^{\mathbb{T}}\overline{V}^*)$
 is globally optimal to problem \eqref{Phi-lam-1}. In fact, if $(\overline{U},\overline{V})$
 is a critical point of \eqref{Phi-lam-1}, then $(P\overline{U},Q\overline{V})$ with $(P,Q)\in\mathbb{O}^{n,m}(M)$
 is a critical point of problem \eqref{Phi-lam}; and conversely, if $(\overline{U},\overline{V})$
 is a critical point of \eqref{Phi-lam}, then $(P^{\mathbb{T}}\overline{U},Q^{\mathbb{T}}\overline{V})$
 is also a critical point of problem \eqref{Phi-lam-1}. Together with Definition \ref{KL-Def1},
 it is not difficult to verify that the following result holds
 for the KL property of $\widetilde{\Phi}_{\lambda}$ and $\Phi_{\lambda}$.
 \begin{lemma}\label{KL-relation}
  Fix any $\lambda>0$. If $\widetilde{\Phi}_{\lambda}$ has the KL property of exponent $1/2$
  at a critical point $(\widetilde{U},\widetilde{V})$ of problem \eqref{Phi-lam-1},
  then $\Phi_{\lambda}$ has the KL property of exponent $1/2$ at
  $(P\widetilde{U},Q\widetilde{V})$ for every $(P,Q)\in\mathbb{O}^{n,m}(M)$.
  Conversely, if $\Phi_{\lambda}$ has the KL property of exponent $1/2$
  at a critical point $(U,V)$ of problem \eqref{Phi-lam},
  then $\widetilde{\Phi}_{\lambda}$ has the KL property of exponent $1/2$
  at $(P^{\mathbb{T}}U,Q^{\mathbb{T}}V)$ for every $(P,Q)\in\mathbb{O}^{n,m}(M)$.
 \end{lemma}

 By Lemma \ref{KL-relation}, to achieve the KL property of exponent $1/2$
 of $\Phi_{\lambda}$ at a global optimal solution of \eqref{Phi-lam}, it suffices
 to establish such a property of $\widetilde{\Phi}_{\lambda}$ at a global optimal
 solution of \eqref{Phi-lam-1}. To this end, we first characterize
 the global optimal solution set of \eqref{Phi-lam-1}.
\begin{lemma}\label{opt-sol}
 Suppose that $\sigma_{r}(\Sigma)>\sigma_{r+1}(\Sigma)$. Then,
 the global optimal solution set of problem \eqref{Phi-lam-1} associated to any given $\lambda>0$
 takes the following form
 \begin{equation}\label{optimal-sol}
  \!\overline{\mathcal{W}}_{\lambda}=\bigg\{\left(\left[\begin{matrix}
                        \overline{U}_{\!1}\\ 0
                        \end{matrix}\right],\left[\begin{matrix}
                        \overline{V}_{\!1}\\ 0
                        \end{matrix}\right]\right)\,\bigg|\
     \overline{U}_1\!=\overline{V}_{\!1}=P(\Sigma_1-\lambda I)_+^{1/2}R
   \ \ {\rm for}\ P\in\mathbb{O}^{r}(\Sigma_1),R\in\mathbb{O}^{r} \!\bigg\}.
 \end{equation}
 \end{lemma}
 \begin{proof}
  By using the expression of $\widetilde{\Phi}_{\lambda}$,
  for any $(U,V)\in\mathbb{R}^{n\times r}\times\mathbb{R}^{m\times r}$
  it holds that
 \begin{align}\label{temp-prob}
  \widetilde{\Phi}_{\lambda}(U,V)
  &\geq\frac{1}{2}\|\sigma(UV^\mathbb{T})-\sigma(\Sigma)\|^2+\lambda\|\sigma(UV^{\mathbb{T}})\|_1\nonumber\\
  &=\frac{1}{2}\|\sigma^{r}(UV^\mathbb{T})-{\rm diag}(\Sigma_1)\|^2+\lambda\|\sigma^{r}(UV^{\mathbb{T}})\|_1
          +\frac{1}{2}\sum_{i=r+1}^n[\sigma_i(\Sigma)]^2\\
  &\ge\min_{z\in\mathbb{R}_+^r}\frac{1}{2}\|z-{\rm diag}(\Sigma_1)\|^2+\lambda\|z\|_1
       +\frac{1}{2}\sum_{i=r+1}^n[\sigma_i(\Sigma)]^2\nonumber
 \end{align}
 where the first inequality is due to the von Neumann's trace inequality,
 and the equality is due to ${\rm rank}(UV^{\mathbb{T}})\le r$. Clearly,
 $x^*=({\rm diag}(\Sigma_1)-\lambda)_+$ is the unique optimal solution
 of the above minimization problem. For any $R\in\mathbb{O}^r$,
 by taking $\overline{U}\!=\big[P(\Sigma_1\!-\!\lambda I)_+^{1/2}R;\ 0\big]$
 and $\overline{V}\!=\big[P(\Sigma_1\!-\!\lambda I)_+^{1/2}R;\ 0\big]$
 with $P\in\mathbb{O}^{r}(\Sigma_1)$, it follows that
 $\widetilde{\Phi}_{\lambda}(\overline{U},\overline{V})$ is equal to
 the optimal value of \eqref{temp-prob}. Thus, the set on the right hand side
 of \eqref{optimal-sol} is included in $\overline{\mathcal{W}}_{\lambda}$.

 For the inverse inclusion, pick a global optimal solution $(\overline{U},\overline{V})$
 of \eqref{Phi-lam-1}. Then, the inequalities in \eqref{temp-prob} necessarily become
 equalities. If not, taking $\overline{U}^*=\overline{P}{\rm Diag}(\sqrt{x^*})$ and
 $\overline{V}^*=\overline{Q}{\rm Diag}(\sqrt{x^*})$ with $(\overline{P},\overline{Q})\in\mathbb{O}^{n,m}(\overline{U}\overline{V}^{\mathbb{T}})$
 yields that $\widetilde{\Phi}_{\lambda}(\overline{U}^*,\overline{V}^*)
 <\widetilde{\Phi}_{\lambda}(\overline{U},\overline{V})$.
 Thus, $\langle\overline{U}\overline{V}^{\mathbb{T}}, \Sigma\rangle
 =\langle\sigma(\overline{U}\overline{V}^{\mathbb{T}}),\sigma(\Sigma)\rangle$,
 which means that $\overline{U}\overline{V}^\mathbb{T}$ and $\Sigma$
 have the same ordered SVD, i.e., there exist $P\in\mathbb{O}^{n}$ and $Q\in\mathbb{O}^{m}$
 such that $\overline{U}\overline{V}^{\mathbb{T}}=P{\rm Diag}(\sigma(\overline{U}\overline{V}^{\mathbb{T}}))Q^{\mathbb{T}}$
 and $\Sigma=P\Sigma Q^\mathbb{T}$. From $\Sigma=P\Sigma Q^\mathbb{T}$,
 one may obtain
 \(
  P=\left(\begin{matrix}
          P_{11} & 0 \\
         0 & P_{22} \\
    \end{matrix}\right)\ \ {\rm and}\ \
  Q=\left(\begin{matrix}
        P_{11} & 0 \\
        0 & Q_{22} \\
     \end{matrix}\right)
  \)
 with $P_{11}\in\mathbb{O}^{r}$ such that $P_{11}\Sigma_1P^\mathbb{T}_{11}=\Sigma_1$.
 Together with $\sigma(\overline{U}\overline{V}^{\mathbb{T}})=x^*$,
 it holds that
 \[
  \left(\begin{matrix}
     \overline{U}_1\overline{V}_1^\mathbb{T} &  \overline{U}_1\overline{V}_2^\mathbb{T} \\
     \overline{U}_2\overline{V}_1^\mathbb{T} & \overline{U}_2\overline{V}_2^\mathbb{T} \\
   \end{matrix}\right)=\overline{U}\overline{V}^\mathbb{T}
   =P{\rm Diag}(x^*)Q^\mathbb{T}
   =\left(\begin{matrix}
             P_{11}{\rm Diag}(x^*)P^\mathbb{T}_{11} & 0 \\
             0 & 0 \\
           \end{matrix}\right),
 \]
 which implies that $\overline{U}_1\overline{V}_1^\mathbb{T}=P_{11}{\rm Diag}(x^*)P^\mathbb{T}_{11}$.
 When $\lambda\geq\sigma_r(\Sigma)$, by Lemma \ref{stationary-lemma2} (ii),
 we have $\overline{U}_2=0$ and $\overline{V}_2=0$. When $\lambda<\sigma_r(\Sigma)$,
 the matrices $\overline{U}_1$ and $\overline{V}_1$ are nonsingular, which along with $\overline{U}_2\overline{V}_1^\mathbb{T}=0$
 and $\overline{U}_1\overline{V}_2^\mathbb{T}=0$ imply $\overline{U}_2=0$ and $\overline{V}_2=0$.
 By Lemma \ref{stationary-lemma1}, $\overline{U}_1^{\mathbb{T}}\overline{U}_1
 = \overline{V}_1^{\mathbb{T}}\overline{V}_1$. Along with
 $\overline{U}_1\overline{V}_1^\mathbb{T}=P_{11}{\rm Diag}(x^*)P^\mathbb{T}_{11}$, we have
 $\overline{U}_1=P_{11}{\rm Diag}(\sqrt{x^*})R=\overline{V}_1$ for some $R\in\mathbb{O}^r$.
 Thus, $(\overline{U},\overline{V})$ belongs to the set on the right hand side of \eqref{optimal-sol}.
 \end{proof}

 Inspired by Lemma \ref{opt-sol} and the proof of \cite[Theorem 2(a)]{ZhangSo18},
 we establish the following conclusion, which extends the result of \cite[Theorem 2(a)]{ZhangSo18}
 to the noisy setting. In fact, as will be shown by Remark \ref{remark41},
 Theorem \ref{KL-Philam} actually implies that $\Phi_{\lambda}$ associated to almost
 all $\lambda>0$ has the KL property of exponent $1/2$ at its global minimizers.
 \begin{theorem}\label{KL-Philam-1}
  Suppose that $\sigma_{r}(\Sigma)>\sigma_{r+1}(\Sigma)$.
  Let $\widetilde{\sigma}_1>\widetilde{\sigma}_2>\cdots>\widetilde{\sigma}_{s}$
  for $1\le s\leq r$ be the distinct singular values of $\Sigma_1$.
  Fix any $\lambda\in(\widetilde{\sigma}_{k+1},\widetilde{\sigma}_k)$ for some $0\leq k\leq s$
  with $\widetilde{\sigma}_0=+\infty$ and $\widetilde{\sigma}_{s+1}=0$.
  Consider any $(\overline{U},\overline{V})\in\overline{\mathcal{W}}_{\lambda}$.
  Then, there exists a constant $\eta>0$ such that
  for all $(U,V)\in\mathbb{B}((\overline{U},\overline{V}),\delta)$ with $
  \delta\leq \min\Big\{\frac{\sqrt{\widetilde{\sigma}_k-\lambda}}{2},
 \frac{\lambda-\widetilde{\sigma}_{k+1}}{2\sqrt{\widetilde{\sigma}_1}}, \frac{\widetilde{\sigma}_{s}-\sigma_{r+1}(\Sigma)}{4\sqrt{\widetilde{\sigma}_1}}\Big\}$,
  \[
   \|\nabla\widetilde{\Phi}_{\lambda}(U,V)\|_F^2\ge
   \eta\big[\widetilde{\Phi}_{\lambda}(U,V)-\widetilde{\Phi}_{\lambda}(\overline{U},\overline{V})\big].
  \]
 \end{theorem}
 \begin{proof}
  Fix any $(U,V)\in\mathbb{B}((\overline{U},\overline{V}),\delta)$. Clearly,
  ${\rm dist}((U,V),\overline{\mathcal{W}}_{\lambda})\le\delta$.
  Pick any $R^*\in\mathop{\arg\min}_{R\in \mathbb{O}^r}\|(U,V)-(\overline{U}\!R,\overline{V}\!R)\|_F$.
  It is easy to check that
  \(
   \widetilde{\Phi}_{\lambda}(\overline{U},\overline{V})=
  \widetilde{\Phi}_{\lambda}(\overline{U}\!R^*,\overline{V}\!R^*).
  \)
  Notice that $\nabla\widetilde{\Phi}_{\lambda}$ is Lipschitz continuous
  on $\mathbb{B}((\overline{U},\overline{V}),\delta)$. By the descent lemma,
 \begin{equation*}
  \widetilde{\Phi}_{\lambda}(U,V)-\widetilde{\Phi}_{\lambda}(\overline{U},\overline{V})
  \le\frac{L}{2}\big(\|U\!-\!\overline{U}\!R^*\|_F^2+\|V\!-\!\overline{V}\!R^*\|_F^2\big)
  =\frac{L}{2}{\rm dist}^2\big((U,V),\mathcal{\overline{W}}_{\lambda}\big)
 \end{equation*}
 where $L>0$ is the Lipschitz constant of $\widetilde{\Phi}_{\lambda}$ on
 $\mathbb{B}((\overline{U},\overline{V}),\delta)$, and the equality is due to
 the characterization of $\mathcal{\overline{W}}_{\lambda}$ in \eqref{optimal-sol}.
 For each $j=1,\ldots,s$, write $a_j:=\big\{i\,|\,\sigma_i(\Sigma_1)=\widetilde{\sigma}_j\big\}$.
 Let $(U_1)_{J}$ and $(\overline{U}_{\!1})_{J}$ be the matrix consisting of
 the first $\widetilde{r}=\sum_{i=1}^k|a_i|$ rows of $U_1$ and $\overline{U}_{\!1}$,
 and let $(U_1)_{\overline{J}}$ and $(\overline{U}_{\!1})_{\overline{J}}$
 be the last $r-\widetilde{r}$ rows of $U_1$ and $\overline{U}_{\!1}$.
 We stipulate $\widetilde{r}=0$ when $k=0$. By Lemma \ref{opt-sol},
 it is immediate to obtain that
 \begin{equation*}
  \!\overline{\mathcal{W}}_{\lambda}\!=\!\bigg\{\left(\left[\begin{matrix}
                        \overline{U}_{\!1}\\ 0
                        \end{matrix}\right],\left[\begin{matrix}
                        \overline{V}_{\!1}\\ 0
                        \end{matrix}\right]\right)\,\bigg|\,
     \overline{U}_{\!1}\!=\overline{V}_{\!1}
     =\!\left(\begin{matrix}
     \widetilde{P}(\Sigma_{11}\!-\!\lambda I)^{\frac{1}{2}}&0\\
         0 & 0
      \end{matrix}\right)\!R
   \ \ {\rm for}\ \widetilde{P}\in\!\mathbb{O}^{\widetilde{r}}(\Sigma_{11}),R\in\mathbb{O}^r\!\bigg\}
 \end{equation*}
 where $\Sigma_{11}={\rm Diag}\big(\sigma_1(\Sigma),\ldots,\sigma_{\widetilde{r}}(\Sigma)\big)
 \in\mathbb{R}^{\widetilde{r}\times\widetilde{r}}$.
 By this expression of $\overline{\mathcal{W}}_{\lambda}$, we have
 \begin{align*}
  &{\rm dist}^2((U,V),\!\overline{\mathcal{W}}_{\lambda})=
    \|U_2\|^2_F+\|V_2\|^2_F\!+\!\min_{R\in\mathbb{O}^r}\{\|U_1-\overline{U}_1R\|_F^2+\|V_1-\overline{V}_1R\|_F^2\}\nonumber\\
  &\le\|U_2\|^2_F\!+\!\|V_2\|^2_F\!+\!2\|U_1\!-\!V_1\|_F^2\!+\!3 {\rm dist}^2(U_1,\overline{U}_1),\nonumber\\
  &=\|U_2\|^2_F\!+\!\|V_2\|^2_F\!+\!2\|U_1\!-\!V_1\|_F^2\!+\!3\|(U_1)_{\overline{J}}\|^2\!+\!
  3{\rm dist}^2((U_1)_{J},[(\Sigma_{11}\!-\!\lambda I)^{\frac{1}{2}}\ \ 0])
 \end{align*}
  where the first inequality is since
  \(
  \|V_1-\overline{V}_{\!1}R\|_F^2\leq2\|V_1-U_1\|_F^2+2\|U_1-\overline{U}_{\!1}R\|_F^2,
  \)
  and the second is due to ${\rm dist}((U_1)_{J},(\overline{U}_{\!1})_{J})
  ={\rm dist}((U_1)_{J},[(\Sigma_{11}\!-\!\lambda I)^{1/2}\ \ 0])$. Thus,
 \begin{align}\label{obj-diffval}
  \widetilde{\Phi}_{\lambda}(U,V)\!-\!\widetilde{\Phi}_{\lambda}(\overline{U},\overline{V})
  \leq&\|U_2\|^2_F+\|V_2\|^2_F+2\|U_1-V_1\|_F^2+3\|(U_1)_{\overline{J}}\|^2\nonumber\\
  &+3{\rm dist}^2((U_1)_{J},[(\Sigma_{11}\!-\!\lambda I)^{1/2}\ \ 0]).
 \end{align}
 Next we bound every term on the right hand side of \eqref{obj-diffval} from above by
 $\|\nabla\widetilde{\Phi}_{\lambda}(U,V)\|_F$.

 \medskip
 \noindent
 {\bf Step 1: bound $\|U_2\|^2_F+\|V_2\|^2_F$ from above by $\|\nabla\widetilde{\Phi}_{\lambda}(U,V)\|_F$.}
 Notice that
 \begin{subnumcases}{}\label{Phi-gradient-U}
  \nabla_1\widetilde{\Phi}_{\lambda}(U,V)
   =(UV^{\mathbb{T}}-\Sigma)V+\lambda{U}
   =\left[\begin{matrix}
    U_{1}(V^\mathbb{T}V+\lambda I)-\Sigma_1V_1 \\
    U_{2}(V^\mathbb{T}V+\lambda I)-\Sigma_2V_2
    \end{matrix}\right],\\
 \label{Phi-gradient-V}
  \nabla_2\widetilde{\Phi}_{\lambda}(U,V)=(UV^{\mathbb{T}}\!-\Sigma)^\mathbb{T}{U}+\lambda{V}
  =\left[\begin{matrix}
  V_{1}(U^\mathbb{T}U+\lambda I)-\Sigma_1U_1 \\
  V_{2}(U^\mathbb{T}U+\lambda I)-\Sigma^{\mathbb{T}}_{2}U_2
  \end{matrix}\right].
  \end{subnumcases}
  By using the second equality of \eqref{Phi-gradient-U} and \eqref{Phi-gradient-V},
  it is not difficult to obtain that
  \begin{align}\label{bound-grad-U2V2}
   &\|\nabla_1\widetilde{\Phi}_{\lambda}(U,V)\|_F+\|\nabla_2\widetilde{\Phi}_2(U,V)\|_F\nonumber\\
   &\geq\|U_{2}(V^\mathbb{T}V+\lambda I)\!-\!\Sigma_2V_2\|_F
       +\|V_{2}(U^\mathbb{T}U+\lambda I)\!-\!\Sigma^{\mathbb{T}}_2U_2\|_F\nonumber\\
   &\geq (\sigma^2_r(V)\!+\!\lambda)\|U_2\|_F\!-\!\|\Sigma_2\|\|V_2\|_F
   +(\sigma^2_r(U)\!+\!\lambda)\|V_2\|_F\!-\!\|\Sigma_2\|\|U_2\|_F\nonumber\\
   &\geq \big[\min(\sigma^2_r(U),\sigma^2_r(V))+\lambda-\|\Sigma_2\|\big]\big(\|U_2\|_F+\|V_2\|_F\big)\nonumber\\
   &\geq \frac{1}{2}[\widetilde{\sigma}_{s}-\sigma_{r+1}(\Sigma)]\big(\|U_2\|_F+\|V_2\|_F\big)
  \end{align}
  where the last inequality is since $\min(\sigma^2_r(U),\sigma^2_r(V))+\lambda-\|\Sigma_2\|
  \ge\widetilde{\sigma}_{s}-\sigma_{r+1}(\Sigma)>0$ implied by $\lambda>\widetilde{\sigma}_{s}$ for $k<s$, and for $k=s$,
  $\min(\sigma^2_r(U),\sigma^2_r(V))+\lambda-\|\Sigma_2\|\geq\frac{1}{2}[\widetilde{\sigma}_{s}-\sigma_{r+1}(\Sigma)]$
  implied by $\sigma_r(U)\!\geq\!\sigma_r(U_1)\!\geq\!\sqrt{\widetilde{\sigma}_s-\lambda}\!-\!\delta$ and
  $\delta\leq\frac{\widetilde{\sigma}_s-\sigma_{r+1}(\Sigma)}{4\sqrt{\widetilde{\sigma}_1}}$.

 \noindent
 {\bf Step 2: bound $\|U_1-V_1\|_F$ from above by $\|\nabla\widetilde{\Phi}_{\lambda}(U,V)\|_F$.}
  From \eqref{Phi-gradient-U} and \eqref{Phi-gradient-V},
 \begin{align*}
   &\|\nabla_1\widetilde{\Phi}_{\lambda}(U,V)\|_F+\|\nabla_2\widetilde{\Phi}_2(U,V)\|_F\nonumber\\
   &\geq \!\|U_{1}(V^\mathbb{T}V+\lambda I)\!-\!\Sigma_1V_1\|_F
         \!+\!\|V_{1}(U^\mathbb{T}U+\lambda I)\!-\!\Sigma_1U_1\|_F\nonumber\\
   &\geq\! \|U_{1}\!(V_1^\mathbb{T}\!V_1\!+\!\lambda I)\!-\!\Sigma_1\!V_1\|_F\!\!
      +\!\|V_{1}(U_1^\mathbb{T}\!U_1\!+\!\lambda I)\!-\!\Sigma_1\!U_1\|_F
     \!\!-\!\delta\|U_{1}\|\|V_2\|_F\!-\!\delta\|V_{1}\|\|U_2\|_F\nonumber\\
   &\geq\! \|U_{1}(V_1^\mathbb{T}\!V_1+\lambda I)\!-\!\Sigma_1\!V_1\|_F
    +\|V_{1}(U_1^\mathbb{T}\!U_1+\lambda I)\!-\!\Sigma_1\!U_1\|_F\nonumber\\
   & \quad -\frac{1}{2}[\widetilde{\sigma}_{s}-\sigma_{r+1}(\Sigma)]\big(\|U_2\|_F+\|V_2\|_F\big)
 \end{align*}
 where the second inequality is using $\max(\|U_2\|_F,\|V_2\|_F)\le
 {\rm dist}((U,V),\mathcal{\overline{W}}_{\lambda})\le\delta$, and the last one
 is due to $\max(\|U_1\|\!,\|V_1\|)
 \!\leq\!\|\overline{U}_{\!1}\|\!+\!\delta\le2\sqrt{\widetilde{\sigma}}_1$ and
 $\delta\!\leq\!\min\Big\{\frac{\lambda-\widetilde{\sigma}_{k+1}}
 {2\sqrt{\widetilde{\sigma}_1}}\!,\frac{\widetilde{\sigma}_{s}-\sigma_{r+1}(\Sigma)}{4\sqrt{\widetilde{\sigma}_1}}\Big\}$.
 By adding this inequality to inequality \eqref{bound-grad-U2V2}, we obtain that
 \begin{align}\label{bound-grad2}
   &2(\|\nabla_1\widetilde{\Phi}_{\lambda}(U,V)\|_F+\|\nabla_2\widetilde{\Phi}_2(U,V)\|_F)\nonumber\\
   &\geq\!\|U_{1}(V_1^\mathbb{T}V_1+\lambda I)
   -\Sigma_1\!V_1\|_F+\|V_{1}(U_1^\mathbb{T}U_1+\!\lambda I)-\Sigma_1U_1\|_F,\\
   &\geq\|U_{1}(V_1^\mathbb{T}V_1+\lambda I)-\Sigma_1V_1-V_{1}(U_1^\mathbb{T}U_1+\lambda I)+\Sigma_1U_1\|_F\nonumber\\
   &=\|(U_{1}-V_1)(V_1^\mathbb{T}V_1+\lambda I)
     +\Sigma_1(U_1-V_1)+V_{1}(V_1^\mathbb{T}V_1-U_1^\mathbb{T}U_1)\|_F\nonumber\\
   &\geq\|(U_{1}-V_1)(V_1^\mathbb{T}V_1+\lambda I)
   +\Sigma_1(U_1-V_1)\|_F-\|V_{1}\|\|V_1^\mathbb{T}V_1-U_1^\mathbb{T}U_1\|_F\nonumber\\
   \label{bound-grad22}
   &\geq\widetilde{\sigma}_s\|U_1-V_1\|_F-2\sqrt{\widetilde{\sigma}_1}\|V_1^\mathbb{T}V_1-U_1^\mathbb{T}U_1\|_F.
 \end{align}
 In addition, from inequality \eqref{bound-grad2} it follows that
 \begin{align*}
   &4\sqrt{\widetilde{\sigma}_1}
   (\|\nabla_1\widetilde{\Phi}_{\lambda}(U,V)\|_F+\|\nabla_2\widetilde{\Phi}_{\lambda}(U,V)\|_F)\nonumber\\
   &\geq\!\|U_{1}\|\|U_{1}(V_1^\mathbb{T}\!V_1\!+\!\lambda I)\!-\!\Sigma_1\!V_1\|_F\!
   +\!\|V_{1}\|\|V_{1}(U_1^\mathbb{T}\!U_1\!+\!\lambda I)\!-\!\Sigma_1\!U_1\|_F\nonumber\\
   &\geq\|U_{1}^{\mathbb{T}}U_{1}(V_1^\mathbb{T}V_1+\lambda I)-U_{1}^{\mathbb{T}}\Sigma_1V_1\|_F
   +\|V_{1}^{\mathbb{T}} V_{1}(U_1^\mathbb{T}U_1+\lambda I)-V_{1}^{\mathbb{T}}\Sigma_1U_1\|_F\nonumber\\
   &=\|U_{1}^{\mathbb{T}}U_{1}(V_1^\mathbb{T}V_1+\lambda I)-U_{1}^{\mathbb{T}}\Sigma_1V_1\|_F
   +\|(U_1^\mathbb{T}U_1+\lambda I)V_{1}^{\mathbb{T}} V_{1}-U_{1}^{\mathbb{T}}\Sigma_1V_1\|_F\\
   &\geq\lambda\|U_{1}^{\mathbb{T}}U_{1}-V_{1}^{\mathbb{T}} V_{1}\|_F.
  \end{align*}
  Substituting this inequality into \eqref{bound-grad22} immediately yields that
  \begin{equation}\label{bound-V1-U1}
    \Big(\frac{8\widetilde{\sigma}_1}{\lambda}+2\Big)\big(\|\nabla_1 \widetilde{\Phi}_{\lambda}(U,V)\|_F
     +\|\nabla_2 \widetilde{\Phi}_{\lambda}(U,V)\|_F\big)\geq\widetilde{\sigma}_s\|U_1-V_1\|_F.
  \end{equation}

 \noindent
 {\bf Step 3: bound $\|(U_1)_{\overline{J}}\|$ from above by $\|\nabla\widetilde{\Phi}_{\lambda}(U,V)\|_F$.}
 Using \eqref{bound-grad2} again, we have
   \begin{align*}
   &2(\|\nabla_1\widetilde{\Phi}_{\lambda}(U,V)\|_F+\|\nabla_2\widetilde{\Phi}_{\lambda}(U,V)\|_F)
     \geq\| V_{1}(U_1^\mathbb{T}U_1+\lambda I)-\Sigma_1U_1\|_F\nonumber\\
   &=\|[U_1U_1^\mathbb{T} -(\Sigma_1-\lambda I)]U_1+(V_{1}-U_1)(U_1^\mathbb{T}U_1+\lambda I)\|_F\nonumber\\
   &\geq\|[U_1U_1^\mathbb{T} -(\Sigma_1-\lambda I)]U_1\|_F-(\sigma_1^2(U_1)+\lambda )\|V_{1}-U_1\|_F
  \end{align*}
  which together with \eqref{bound-V1-U1} and $\|U_1\|
 \!\leq\!\|\overline{U}_{\!1}\|\!+\!\delta\le2\sqrt{\widetilde{\sigma}}_1$ implies that
 \begin{align}\label{bound-U1U1}
  &\Big[\frac{(4\widetilde{\sigma}_1+\lambda )}{\widetilde{\sigma}_s}
   \Big(\frac{8\widetilde{\sigma}_1}{\lambda}+2\Big)+2\Big]\big[\|\nabla_1 \widetilde{\Phi}_{\lambda}(U,V)\|_F
   +\|\nabla_2 \widetilde{\Phi}_{\lambda}(U,V)\|_F\big]\nonumber\\
   &\geq\|[U_1U_1^\mathbb{T} -(\Sigma_1-\lambda I)]U_1\|_F.
  \end{align}
  From the definitions of the index sets $J$ and $\overline{J}$, it follows that
  \begin{align}\label{temp-bound-equa}
   [U_1U_1^\mathbb{T} -(\Sigma_1-\lambda I)]U_1
   =\left[\begin{matrix}
  (U_1)_{J}(U_1^\mathbb{T}U_1)-(\Sigma_{11}-\lambda I)(U_1)_{J}  \\
  (U_1)_{\overline{J}}(U_1^\mathbb{T}U_1)-(\Sigma_{12}-\lambda I)(U_1)_{\overline{J}} \\
  \end{matrix}\right],
  \end{align}
 where $\Sigma_{12}={\rm Diag}\big(\sigma_{\widetilde{r}+1}(\Sigma),\ldots,\sigma_{r}(\Sigma)\big)
 \in\mathbb{R}^{(r-\widetilde{r})\times(r-\widetilde{r})}$.
 Thus, it holds that
 \begin{align}\label{bound-UI}
  \|[U_1U_1^\mathbb{T} -(\Sigma_1-\lambda I)]U_1\|_F
  &\ge\|(U_1)_{\overline{J}}(U_1^\mathbb{T}U_1)-(\Sigma_{12}-\lambda I)(U_1)_{\overline{J}}\|_F\nonumber \\
    &\ge\lambda\|(U_1)_{\overline{J}}\|_F-\|\Sigma_{12}(U_1)_{\overline{J}}\|_F\nonumber\\
    &\geq(\lambda-\widetilde{\sigma}_{k+1})\|(U_1)_{\overline{J}}\|_F.
 \end{align}
 Combining this inequality with \eqref{bound-U1U1}, we obtain the desired result.

 \noindent
 {\bf Step 4: bound ${\rm dist}((U_1)_{J},[(\Sigma_{11}\!-\!\lambda I)^{1/2}\ \ 0])$ from above
 by $\|\nabla\widetilde{\Phi}_{\lambda}(U,V)\|_F$.} From \eqref{temp-bound-equa} and
 $\|(U_1)_{\overline{J}}\|_F =\|(U_1)_{\overline{J}}-(\overline{U}_{\!1})_{\overline{J}}\|_F
 \leq {\rm dist}\big((U,V),\!\overline{\mathcal{W}}_{\lambda}\big)\leq \delta$,
 it follows that
 \begin{align}\label{U1U1-D}
  \|[U_1U_1^\mathbb{T}\!-(\Sigma_1\!-\lambda I)]U_1\|_F
  &\geq\|(U_1)_J(U_1^\mathbb{T}U_1)-(\Sigma_{11}-\lambda I)(U_1)_J\|_F\nonumber\\
  &\geq\|(U_1)_J[(U_1)_{J}^\mathbb{T}(U_1)_{J}+(U_1)_{\overline{J}}^\mathbb{T}(U_1)_{\overline{J}}]
       -(\Sigma_{11}-\lambda I)(U_1)_J\|_F\nonumber\\
  &\geq\|\big((U_1)_{J}(U_1)_{J}^\mathbb{T}-(\Sigma_{11}-\lambda I)\big)(U_1)_{J}\|_F-\|(U_1)_{J}\|
         \|(U_1)_{\overline{J}}\|_F^2\nonumber\\
  &\geq\sigma_{\widetilde{r}}[(U_1)_{J}]
      \|(U_1)_{J}\!(U_1)_{J}^\mathbb{T}\!-\!(\Sigma_{11}\!-\!\lambda I)\!\|_F\!
      -\!2\delta\sqrt{\widetilde{\sigma}_1}\|\!(U_1)_{\overline{J}}\!\|_F
 \end{align}
  where the last inequality is using $\|AB\|_F\geq\sigma_{\widetilde{r}}(B)\|A\|_F$ for
  $A\in\mathbb{R}^{\widetilde{r}\times\widetilde{r}}$ and $B\in\mathbb{R}^{\widetilde{r}\times r}$,
  and $\|(U_1)_{J}\|\leq\|(\overline{U}_1)_{J}\|+\delta\leq2\sqrt{\widetilde{\sigma}_1}$.
  Now we bound $\sigma_{\widetilde{r}}((U_1)_{J})$ from below.
  Let $\widetilde{R}\in\mathbb{O}^{\widetilde{r}}$ be such that
  ${\rm dist}((U_1)_{J},[(\Sigma_{11}-\lambda I)^{1/2}\ \ 0])
  =\|(U_1)_{J}-[(\Sigma_{11}-\lambda I)^{1/2}\ \ 0]\widetilde{R}\|_F$. Then,
  \begin{align}\label{UI-rSV}
  \sigma_{\widetilde{r}}((U_1)_{J})&=\min_{\|x\|=1,\|y\|=1}|x^{\mathbb{T}}(U_1)_{J}y|\nonumber\\
  &\ge \min_{\|x\|=1,\|y\|=1}\big\|x^{\mathbb{T}}[(\Sigma_{11}-\lambda I)^{1/2}\ \ 0]\widetilde{R}y\big\|_2
       \nonumber\\
  &\qquad -\max_{\|x\|=1,\|y\|=1}\big\|x^{\mathbb{T}}\big[(U_1)_{J}
          -[(\Sigma_{11}-\lambda I)^{1/2}\ \ 0]\widetilde{R}\big]y\big\|_2\nonumber\\
  &\ge \sqrt{\widetilde{\sigma}_k-\lambda}
        -\|(U_1)_{J}-[(\Sigma_{11}-\lambda I)^{1/2}\ \ 0]\widetilde{R}\|_F\nonumber\\
  &=\sqrt{\widetilde{\sigma}_k-\lambda}-{\rm dist}\big((U_1)_{J},[(\Sigma_{11}-\lambda I)^{1/2}\ \ 0]\big)
    \ge\frac{1}{2}\sqrt{\widetilde{\sigma}_k-\lambda}
 \end{align}
 where the last inequality is due to
 ${\rm dist}((U_1)_{J},\big[(\Sigma_{11}\!-\!\lambda I)^{1/2}\ \ 0]\big)\leq\delta\le\frac{\sqrt{\widetilde{\sigma}_k-\lambda}}{2}$.
 Now by invoking equation \eqref{bound-UI}-\eqref{UI-rSV} and $\delta\leq\frac{\lambda-\widetilde{\sigma}_{k+1}}{2\sqrt{\widetilde{\sigma}_1}}$,
 it follows that
 \begin{align*}
  &2\|[U_1U_1^\mathbb{T} -(\Sigma_1-\lambda I)]U_1\|_F
  \ge\frac{1}{2}\sqrt{\widetilde{\sigma}_k-\lambda}
    \big\|(U_1)_{J}(U_1)_{J}^\mathbb{T}-(\Sigma_{11}\!-\lambda I)\big\|_F\nonumber\\
  &\geq\frac{1}{2}(\widetilde{\sigma}_k-\lambda)^{3/2}
    \big\|(\Sigma_{11}\!-\lambda I)^{-1/2}(U_1)_{J}(U_1)_{J}^\mathbb{T}
    (\Sigma_{11}-\lambda I)^{-1/2}-I\big\|_F.
  \end{align*}
  Let $(\Sigma_{11}\!-\lambda I)^{-1/2}(U_1)_J$ have the SVD as
  $(\Sigma_{11}\!-\lambda I)^{-1/2}(U_1)_{J}=L[\Lambda\ \ 0]H^\mathbb{T}$,
  where $L\in\mathbb{O}^{\widetilde{r}}$ and $H\in\mathbb{O}^r$.
  Take $\overline{L}=\left(\begin{matrix}
                                    L & 0 \\
                                    0 & I \\
                                 \end{matrix}
                    \right)\in\mathbb{R}^{r\times r}$.
 Clearly, $\overline{L}^\mathbb{T}\overline{L}=\overline{L}\,\overline{L}^\mathbb{T}=I$.
 Then, it holds that
 \begin{align*}
  &\big\|(\Sigma_{11}\!-\lambda I)^{-1/2}(U_1)_{J}(U_1)_J^\mathbb{T}
          (\Sigma_{11}\!-\lambda I)^{-1/2} - I\big\|_F\nonumber\\
   &=\big\|\Lambda^2-I\|_F=\|(\Lambda+ I)(\Lambda- I)\big\|_F\nonumber\\
   &\geq\|[\Lambda\ \ 0] -[I\ \ 0]\|_F=\|(\Sigma_{11}\!-\lambda I)^{-1/2}(U_1)_{J}H- L[I\ \ 0]\|_F\nonumber\\
   &\geq\frac{1}{\sqrt{\widetilde{\sigma}_1-\lambda}}\big\|(U_1)_{J}H
      -[(\Sigma_{11}\!-\lambda I)^{1/2}L\ \ 0]\big\|_F\nonumber\\
   &=\frac{1}{\sqrt{\widetilde{\sigma}_1-\lambda}}\big\|(U_1)_{J}H-[(\Sigma_{11}\!-\lambda I)^{1/2}\ \ 0]\overline{L}\big\|_F\\
   &=\frac{1}{\sqrt{\widetilde{\sigma}_1-\lambda}}\big\|(U_1)_{J}-
        [(\Sigma_{11}-\lambda I)^{1/2}\ \ 0]\overline{L}H^\mathbb{T}\big\|_F\nonumber\\
   &\geq\frac{1}{\sqrt{\widetilde{\sigma}_1-\lambda}}
       {\rm dist}\big((U_1)_{J},[(\Sigma_{11}\!-\lambda I)^{1/2}\ \ 0]\big).
 \end{align*}
 By combining the last two inequalities with \eqref{bound-U1U1}, it is immediate to obtain that
  \begin{align*}
   &2\Big[\frac{(4\widetilde{\sigma}_1+\lambda )}
          {\widetilde{\sigma}_s}\Big(\frac{8\widetilde{\sigma}_1}{\lambda}+2\Big)+2\Big]
   \big(\|\nabla_1 \widetilde{\Phi}_{\lambda}(U,V)\|_F+\|\nabla_2 \widetilde{\Phi}_{\lambda}(U,V)\|_F\big)\nonumber\\
   &\ge\frac{(\widetilde{\sigma}_k-\lambda)^{3/2}}{2\sqrt{\widetilde{\sigma}_1-\lambda}}
      {\rm dist}\big((U_1)_{J},\big[(\Sigma_{11}\!-\lambda I)^{1/2}\ \ 0\big]\big).
  \end{align*}
 Now from the result in the above four steps and \eqref{obj-diffval},
 we obtain the conclusion.
 \end{proof}
 \begin{remark}\label{remark41}
 {\bf(i)} When $\lambda>\sigma_r(\Sigma)$, inequality \eqref{bound-grad-U2V2}
can be replaced by the following one
 \[
  \|\nabla_1\widetilde{\Phi}_{\lambda}(U,V)\|_F+\|\nabla_2\widetilde{\Phi}_2(U,V)\|_F\geq
  (\lambda-{\sigma}_{r+1}(\Sigma))\big(\|U_2\|_F+\|V_2\|_F\big),
 \]
 and the conditions $\sigma_{r}(\Sigma)>{\sigma}_{r+1}(\Sigma)$
 and $\delta\leq\frac{\widetilde{\sigma}_{s}-\sigma_{r+1}(\Sigma)}{4\sqrt{\widetilde{\sigma}_1}}$
 in Theorem \ref{KL-Philam-1} can be removed.

  \medskip
  \noindent
  {\bf(ii)}
  When $k=s$ and $E=0$, Theorem \ref{KL-Philam-1} implies that $\widetilde{\Phi}_{\lambda}$
  with $0<\lambda<\sigma_r(\Sigma)$ has the KL property of exponent $1/2$
  at every global minimizer, which is precisely the result of \cite[Theorem 2(a)]{ZhangSo18}.
  When $k=0$, Theorem \ref{KL-Philam-1} implies that $\widetilde{\Phi}_{\lambda}$
  with $\lambda>\sigma_1(\Sigma)$ has the KL property of exponent $1/2$
  at every global minimizer. By Lemma \ref{KL-relation},
  $\Phi_{\lambda}$ associated to the corresponding $\lambda$ also has
  the KL property of exponent $1/2$ at its global minimizers.
 \end{remark}

 \subsection{Noisy and partial sampling}\label{sec4.2}

 In this scenario, the function $f$ is given by \eqref{Lsquare}.
 For each $\lambda>0$, denote by $\mathcal{S}_{\lambda}$
 the critical point set of $\Phi_{\lambda}$. Define
 $\Upsilon_1\!:\mathbb{R}^{n \times r}\times\mathbb{R}^{m\times r}\!\to\mathbb{R}^{n\times r}$
 and $\Upsilon_2\!:\mathbb{R}^{n\times r}\times\mathbb{R}^{m\times r}\!\to\mathbb{R}^{m\times r}$ by
 \begin{equation}\label{Upsilon12}
  \Upsilon_1(U,V):=[\mathcal{A}^*\mathcal{A}(UV^\mathbb{T}\!-\!M)]V\ \ {\rm and}\ \
  \Upsilon_2(U,V):=[\mathcal{A}^*\mathcal{A}(UV^\mathbb{T}\!-\!M)]^\mathbb{T}U.
 \end{equation}
 The following theorem states that $\Phi_{\lambda}$ has the KL property of
 exponent $1/2$ at a critical point $(\overline{U},\overline{V})$ if
 the calmness modulus of $\Upsilon_1$ and $\Upsilon_2$ at $(\overline{U},\overline{V})$
 is not too greater than $\lambda$.
 \begin{theorem}\label{KL-Philam}
  Fix any $\lambda>0$. Consider any point $(\overline{U},\overline{V})\in\mathcal{S}_{\lambda}$.
  Suppose that there exists $\varepsilon>0$ such that the calmness modulus of
  $\Upsilon_1$ and $\Upsilon_2$ on $\mathbb{B}((\overline{U},\overline{V}),\varepsilon)$,
  say $c_1$ and $c_2$, satisfies $2\overline{c}+\|\mathcal{A}^*(\omega)\|
   +\!\sqrt{(2\overline{c}+\|\mathcal{A}^*(\omega)\|)^2+4\overline{c}\|\mathcal{A}^*(\omega)\|}<\lambda$,
  where $\overline{c}=\max(c_1,c_2)$. Then, the function $\Phi_{\lambda}$ has the KL property of
  exponent $1/2$ at $(\overline{U},\overline{V})$.
 \end{theorem}
 \begin{proof}
  By the definition of calmness in \cite[Chapter 8F]{RW98},
  for any $(U,V)\in\mathbb{B}((\overline{U},\overline{V}),\varepsilon)$,
  \begin{subnumcases}{}\label{calmness1}
   \|\Upsilon_1(U,V)-\Upsilon_1(\overline{U},\overline{V})\|_F\le
    c_1\|(U,V)-(\overline{U},\overline{V})\|_F,\\
    \label{calmness2}
   \|\Upsilon_2(U,V)-\Upsilon_2(\overline{U},\overline{V})\|_F\le
    c_2\|(U,V)-(\overline{U},\overline{V})\|_F.
 \end{subnumcases}
 Observe that $\nabla{\Phi}_{\lambda}$ is globally Lipschitz continuous on
 $\mathbb{B}((\overline{U},\overline{V}),\varepsilon)$. Then, there exists
 a constant $L>0$ such that
  for all $(U,V),(U',V')\in\mathbb{B}((\overline{U},\overline{V}),\varepsilon)$,
 \begin{equation}\label{obj-val1}
  {\Phi}_{\lambda}(U,V)-{\Phi}_{\lambda}(\overline{U},\overline{V})
  <\frac{L}{2}\big(\|U\!-\!\overline{U}\|_F^2+\|V\!-\!\overline{V}\|_F^2\big).
 \end{equation}
 Pick any $(U,V)$ from $\mathbb{B}((\overline{U},\overline{V}),\varepsilon)$.
 Notice that $\nabla_1\Phi_{\lambda}(\overline{U},\overline{V})=0$.
 Then, it holds that
 \begin{align*}
  &\|\nabla_1\Phi_{\lambda}(U,V)\|^2_F
    =\big\|\mathcal{A}^*(\mathcal{A}(UV^\mathbb{T}\!-\!M)-\omega)V+\lambda U-
      \mathcal{A}^*(\mathcal{A}(\overline{U}\overline{V}^\mathbb{T}\!-\!M)-\omega)\overline{V}
      -\lambda \overline{U}\big\|_F^2\\
  &= \big\|\mathcal{A}^*(\mathcal{A}(UV^\mathbb{T}\!-\!M))V
         -\mathcal{A}^*(\mathcal{A}(\overline{U}\overline{V}^\mathbb{T}\!-\!M))\overline{V} -\mathcal{A}^*(\omega)(V-\overline{V}) +\lambda (U-\overline{U})\big\|^2_F\\
  &=\|\mathcal{A}^*(\mathcal{A}(UV^\mathbb{T}\!-\!M))V-\mathcal{A}^*(\mathcal{A}(\overline{U}\overline{V}^\mathbb{T}\!-\!M))\overline{V}\|^2_F+\|\mathcal{A}^*(\omega)(V-\overline{V})\|_F^2\\
  &\quad+\lambda^2 \|U-\overline{U}\|_F^2-\!2\lambda\langle\mathcal{A}^*(\omega)(V\!-\!\overline{V}), (U\!-\!\overline{U})\rangle\\
  &\quad-2\langle\Upsilon_1(U,V)-\Upsilon_1(\overline{U},\overline{V})
     ,\mathcal{A}^*(\omega)(V-\overline{V})-\lambda (U\!-\!\overline{U})\rangle.
 \end{align*}
  Similarly, by the expression of $\nabla_2\Phi_{\lambda}(U,V)$ and
  the fact that $\nabla_2\Phi_{\lambda}(\overline{U},\overline{V})=0$, we have
 \begin{align*}
  \|\nabla_2\Phi_{\lambda}(U,V)\|^2_F
   &=\big\|[\mathcal{A}^*\mathcal{A}(UV^\mathbb{T}\!-\!M)]^\mathbb{T}U
      -[\mathcal{A}^*\mathcal{A}(\overline{U}\overline{V}^\mathbb{T}\!-\!M)]^\mathbb{T}\overline{U}\|^2_F\\
  &\quad+\|[\mathcal{A}^*(\omega)]^\mathbb{T}(U-\overline{U})\|_F^2
   +\lambda^2 \|V-\overline{V}\|_F^2-2\lambda\langle\mathcal{A}^*(\omega)(V-\overline{V}), U-\overline{U}\rangle\\
  &\quad-2\langle\Upsilon_2(U,V)-\Upsilon_2(\overline{U},\overline{V}),
      [\mathcal{A}^*(\omega)]^{\mathbb{T}}(U-\overline{U})-\lambda (V-\overline{V})\rangle.
 \end{align*}
  From the above two equalities, it immediately follows that
 \begin{align}\label{gradPhi}
  &\|\nabla\Phi_{\lambda}(U,V)\|^2_F=\|\nabla_1\Phi_{\lambda}(U,V)\|^2_F+\|\nabla_2\Phi_{\lambda}(U,V)\|^2_F\nonumber\\
  &\ge\lambda^2 \|U-\overline{U}\|_F^2+\lambda^2 \|V-\overline{V}\|_F^2
      -4\lambda\langle\mathcal{A}^*(\omega)(V-\overline{V}),U-\overline{U}\rangle\nonumber\\
  &\quad-2\langle\Upsilon_1(U,V)-\Upsilon_1(\overline{U},\overline{V}),
     \mathcal{A}^*(\omega)(V-\overline{V})-\lambda (U\!-\!\overline{U})\rangle\nonumber\\
  &\quad-2\langle\Upsilon_2(U,V)-\Upsilon_2(\overline{U},\overline{V}),[\mathcal{A}^*(\omega)]^{\mathbb{T}}(U-\overline{U})-\lambda (V-\overline{V})\rangle,\nonumber\\
 &\ge\lambda^2 \|U-\overline{U}\|_F^2+\lambda^2 \|V-\overline{V}\|_F^2-4\lambda\|\mathcal{A}^*(\omega)\|\|U-\overline{U}\|_F\|V-\overline{V}\|_F\nonumber\\
  &\quad-\underbrace{2\langle\Upsilon_1(U,V)-\Upsilon_1(\overline{U},\overline{V}),
     \mathcal{A}^*(\omega)(V-\overline{V})-\lambda (U\!-\!\overline{U})\rangle}_{I_1}\nonumber\\
  &\quad-\underbrace{2\langle\Upsilon_2(U,V)-\Upsilon_2(\overline{U},\overline{V}),[\mathcal{A}^*(\omega)]^{\mathbb{T}}(U-\overline{U})-\lambda (V-\overline{V})\rangle}_{I_2}.
 \end{align}
 Next, we separately bound the above $I_1$ and $I_2$. For the term $I_1$, it holds that
 \begin{align*}
    I_1\leq2\|\Upsilon_1(U,V)-\Upsilon_1(\overline{U},\overline{V})\|_F
    \big(\|\mathcal{A}^*(\omega)\|\|V-\overline{V}\|_F+\lambda\|U\!-\!\overline{U}\|_F\big),
 \end{align*}
 which together with \eqref{calmness1} implies that
 \begin{align*}
    I_1&\leq2c_1\|(U,V)-(\overline{U},\overline{V})\|_F\big(\|\mathcal{A}^*(\omega)\|\|V-\overline{V}\|_F+\lambda\|U\!-\!\overline{U}\|_F\big)\\
       &\leq2c_1\big(\|U-\overline{U}\|_F+\|V-\overline{V}\|_F\big)\big(\|\mathcal{A}^*(\omega)\|\|V-\overline{V}\|_F+\lambda\|U\!-\!\overline{U}\|_F\big)\\
       &\!=\!2\lambda c_1\|U\!\!-\!\overline{U}\|^2_F
       \!+\!2c_1\|\mathcal{A}^*(\omega)\|\|V\!\!-\!\overline{V}\|_F^2
       \!+\!2c_1\big(\lambda\!+\!\|\mathcal{A}^*(\omega)\|\big)\|U\!-\!\overline{U}\|_F\|V\!\!-\!\overline{V}\|_F.
 \end{align*}
 Similarly, for the term $I_2$ in \eqref{gradPhi}, following the same analysis
 and using \eqref{calmness2} yield that
 \begin{align*}
  I_2\!\le 2c_2\|\mathcal{A}^*(\omega)\|\|U\!\!-\!\overline{U}\|^2_F+2\lambda c_2\|V\!\!-\!\overline{V}\|_F^2
       \!+\!2c_2\big(\lambda\!+\!\|\mathcal{A}^*(\omega)\|\big)\|U\!-\!\overline{U}\|_F\|V\!\!-\!\overline{V}\|_F
 \end{align*}
 By combining the last two inequalities with \eqref{gradPhi}, we obtain that
\begin{align*}
  \|\nabla\Phi_{\lambda}(U,V)\|^2_F
  &\ge\lambda^2 \|U-\overline{U}\|_F^2+\lambda^2 \|V-\overline{V}\|_F^2
      -2(\lambda c_1+c_2\|\mathcal{A}^*(\omega)\|)\|U\!\!-\!\overline{U}\|^2_F\nonumber\\
  &\quad-2(\lambda c_2+c_1\|\mathcal{A}^*(\omega)\|)\|V\!\!-\!\overline{V}\|^2_F
        -4\lambda\|\mathcal{A}^*(\omega)\|\|U-\overline{U}\|_F\|V-\overline{V}\|_F\nonumber\\
  &\quad-2(c_1+c_2)(\lambda\!+\!\|\mathcal{A}^*(\omega)\|)\|U-\overline{U}\|_F\|V-\overline{V}\|_F.
 \end{align*}
 This together with the basic inequality $2ab\leq a^2+b^2$ for any $a,b\in\mathbb{R}$ implies that
  \begin{align}\label{gradPhi1}
  \|\nabla\Phi_{\lambda}(U,V)\|^2_F
  &\ge\Gamma_1(\lambda) \|U-\overline{U}\|_F^2+\Gamma_2(\lambda) \|V-\overline{V}\|_F^2
 \end{align}
 where
 \begin{align*}
  &\Gamma_1(\lambda):=\lambda^2-2(\lambda c_1+c_2\|\mathcal{A}^*(\omega)\|)-2\lambda\|\mathcal{A}^*(\omega)\|
     -(c_1+c_2)(\lambda\!+\!\|\mathcal{A}^*(\omega)\|);\\
  &\Gamma_2(\lambda):=\lambda^2-2(\lambda c_2+c_1\|\mathcal{A}^*(\omega)\|)-2\lambda\|\mathcal{A}^*(\omega)\|
     -(c_1+c_2)(\lambda\!+\!\|\mathcal{A}^*(\omega)\|).
 \end{align*}
 Recall that
 \(
  2\overline{c}+\|\mathcal{A}^*(\omega)\|
   +\sqrt{(2\overline{c}+\|\mathcal{A}^*(\omega)\|)^2+4\overline{c}\|\mathcal{A}^*(\omega)\|}<\lambda.
 \)
 We have $\Gamma_1(\lambda)>0$ and $\Gamma_2(\lambda)>0$.
 By comparing \eqref{gradPhi1} with \eqref{obj-val1}, there exists a constant
 $\eta>0$ such that for all $(U,V)\in\mathbb{B}((\overline{U},\overline{V}),\varepsilon)$,
 $\|\nabla\Phi_{\lambda}(U,V)\|_F\ge\eta\sqrt{\Phi_{\lambda}(U,V)-\Phi_{\lambda}(\overline{U},\overline{V})}$.
\end{proof}
\begin{remark}\label{remark42}
 Suppose that $(\overline{U},\overline{V})$ is a global minimizer of $\Phi_{\lambda}$
 with $\overline{c}\ll\|\mathcal{A}^*(\omega)\|$. Then, the condition that
 $2\overline{c}+\|\mathcal{A}^*(\omega)\|  +\!\sqrt{(2\overline{c}+\|\mathcal{A}^*(\omega)\|)^2+4\overline{c}\|\mathcal{A}^*(\omega)\|}<\lambda$
 approximately requires $\lambda>2\|\mathcal{A}^*(\omega)\|$. Such $\lambda$ is
 close to the optimal one obtained in \cite[Corollary 1]{Negahban11} for the error bound
 of the optimal solution to the nuclear norm regularized problem.
\end{remark}

 Theorem \ref{KL-Philam} states that $\Phi_{\lambda}$ has the KL property
 of exponent $1/2$ only at its part of critical points. In fact,
 even in the noiseless and full sampling setup, $\Phi_{\lambda}$ with
 some $\lambda$ does not have the KL property
 of exponent $1/2$ at its critical points; see Example \ref{counter-example}.
 \begin{example}\label{counter-example}
  Pick an arbitrary $a>0$. Consider $r=2$ and $\Sigma=a I$.
  Fix any $\lambda<a$. Take $(\overline{U},\overline{V})\in\mathbb{R}^{n\times 2}\times\mathbb{R}^{m\times 2}$
  with $\overline{U}_1=\overline{V}_1={\rm Diag}(0,\sqrt{d})$
  and $(\overline{U}_2,\overline{V}_2)=(0,0)$ for $d=a-\lambda$.
  It is easy to check that $(\overline{U},\overline{V})$ is
  a critical point of $\widetilde{\Phi}_{\lambda}$ with
  $\widetilde{\Phi}_{\lambda}(\overline{U},\overline{V})
 =\frac{a^2+\lambda^2}{2}+\lambda(a-\lambda)$.
 For each $k\in\mathbb{N}$, let $(U^k,V^k)\in\mathbb{R}^{n\times 2}\times\mathbb{R}^{m\times 2}$
 with $U_1^k=V_1^k=\left(\begin{matrix}
                  0 & \frac{1}{k^2} \\
                  \frac{1}{k^2} &\sqrt{d}+\frac{1}{k^4} \\
                \end{matrix}\right)$ and $(U_2^k,V_2^k)=0$.
 Clearly, $\|(U^k,V^k)-(\overline{U},\overline{V})\|_F=O(\frac{1}{k^2})$.
 After a simple calculation,
 \[
  (U_1^k)^{\mathbb{T}}U_1^k-\Sigma
  =\left(\begin{matrix}
   \frac{1}{k^4}-a & \frac{\sqrt{d}}{k^2} +\frac{1}{k^6} \\
   \frac{\sqrt{d}}{k^2}+\frac{1}{k^6}&\frac{1+2\sqrt{d}}{k^4} +\frac{1}{k^8}-\lambda  \\
   \end{matrix}\right),
 \]
 and consequently,
 \(
   \widetilde{\Phi}_{\lambda}(U^k,V^k)=\frac{a^2+\lambda^2}{2}+\lambda(a-\lambda)+O(\frac{1}{k^4}).
  \)
 Thus, we obtain that
 \[
   \widetilde{\Phi}_{\lambda}(U^k,V^k)-\Phi_{\lambda}(\overline{U},\overline{V})=O(\frac{1}{k^4}).
 \]
 On the other hand, by the expression of $\nabla\Phi_{\lambda}(U,V)$, it is not difficult to calculate that
 \begin{align*}
  &\|\nabla\widetilde{\Phi}_{\lambda}(U^k,V^k)\|_F^2
  =2\Big\|U_1^k((U_1^k)^{\mathbb{T}}U_1^k-dI)\|_F^2\\
  &=2\bigg\|\left(\begin{matrix}
    \frac{\sqrt{d}}{k^4}+\frac{1}{k^8} & \frac{1+2\sqrt{d}}{k^6} +\frac{1}{k^{10}} \\
    \frac{1+2\sqrt{d}}{k^6} +\frac{1}{k^{10}} & \frac{2{d}+\sqrt{2d}}{k^4} +\frac{2+3\sqrt{d}}{k^8} +\frac{1}{k^{12}}\\
    \end{matrix}\right)\bigg\|^2_F=O\big(\frac{1}{k^8}\big).
 \end{align*}
 The last two equations show that $\widetilde{\Phi}_{\lambda}$ associated to $\lambda<a$ does not
 have the KL property of exponent $1/2$ at the critical point $(\overline{U},\overline{V})$.
 \end{example}

 For Example \ref{counter-example}, it is easy to calculate that
 the calmness modulus of $\Upsilon_1$ in \eqref{Upsilon12} is at least $2a-\lambda$.
 Clearly, the condition of Theorem \ref{KL-Philam} is not satisfied.

 \section{Numerical experiments}\label{sec5}

 We confirm the previous theoretical findings by applying an accelerated alternating
 minimization (AAL) method to problem \eqref{Phi-lam}. Fix any $\lambda>0$ and any
 $(U^0,V^0)\in\mathbb{R}^{n\times r}\times\mathbb{R}^{m\times r}$. Write
 \(
   \mathcal{L}_{\lambda,0}:=\big\{(U,V)\in\mathbb{R}^{n\times r}\times\mathbb{R}^{m\times r}
   \ |\ \Phi_{\lambda}(U,V)\le\Phi_{\lambda}(U^0,V^0)\big\}.
 \)
 Since $\Phi_{\lambda}$ is coercive, the set $\mathcal{L}_{\lambda,0}$
 is nonempty and compact. Let $F(U,V):=f(UV^{\mathbb{T}})$ for
 $(U,V)\in\mathbb{R}^{n\times r}\times\mathbb{R}^{m\times r}$.
 Clearly, for any given $(U,V)\in\mathcal{L}_{\lambda,0}$,
 the functions $\nabla_{1}F(\cdot,V)$ and $\nabla_{2}F(U,\cdot)$
 are globally Lipschitz continuous on $\mathcal{L}_{\lambda,0}$,
 and their Lipschitz constants, say $L_U$ and $L_V$, depend on $\|\nabla^2F(U,V)\|$, which is bounded
 on the set $\mathcal{L}_{\lambda,0}$. Hence, there exists a constant $L_F>0$ such that
 $\max(L_U,L_V)\le L_F$ for all $(U,V)\in\mathcal{L}_{\lambda,0}$.

 Now we describe the iterate steps of the AAL method for solving problem \eqref{Phi-lam}.
 \begin{algorithm}[!h]
  \caption{\label{AAL}{\bf(AAL method for solving problem \eqref{Phi-lam})}}
  \textbf{Initialization:} Choose an appropriate $\lambda>0$, an integer $r\ge 1$
  and $\beta_0\in\big[0,\sqrt{\frac{L}{L+L_F}}\big]$ with $L\ge L_F$, and
  $(U^0,V^0)\in\mathbb{R}^{n\times r}\times\mathbb{R}^{m\times r}$.
  Set $(U^{-1},V^{-1})=(U^0,V^0)$ and $k=0$.

  \medskip
  \noindent
 \textbf{while} the stopping conditions are not satisfied \textbf{do}
  \begin{itemize}
   \item Set $\widetilde{U}^{k}=U^{k}+\beta_k(U^{k}\!-\!U^{k-1})$ and
         $\widetilde{V}^{k}=V^{k}+\beta_k(V^{k}\!-\!V^{k-1})$;

   \item Solve the following two minimization problems
         \begin{subequations}
         \begin{align}\label{AL-U}
          U^{k+1}\in\mathop{\arg\min}_{U\in\mathbb{R}^{n\times r}}
           \Big\{\langle\nabla_{1}F(\widetilde{U}^k,{V}^k),U\!-\!\widetilde{U}^k\rangle
           +\frac{L_F}{2}\|U\!-\!\widetilde{U}^k\|_F^2+\frac{\lambda}{2}\|U\|_F^2\Big\},\\
         \label{AL-V}
          V^{k+1}\in\mathop{\arg\min}_{V\in\mathbb{R}^{m\times r}}
           \Big\{\langle\nabla_{2}F(U^{k+1},\widetilde{V}^k),V\!-\!\widetilde{V}^k\rangle
           +\frac{L_F}{2}\|V\!-\!\widetilde{V}^k\|_F^2+\frac{\lambda}{2}\|V\|_F^2\Big\}.
         \end{align}
         \end{subequations}

   \item Update $\beta_{k}$ to be $\beta_{k+1}$ such that $\beta_{k+1}\in\big[0,\sqrt{\frac{L}{L+L_F}}\big]$.
  \end{itemize}
 \textbf{end while}
 \end{algorithm}
 \begin{remark}
  {\bf(i)} Algorithm \ref{AAL} is the special case of \cite[Algorithm 1]{XuYin13}
  with $s=2$. Since $\Phi_{\lambda}$ is semialgebraic, by following the analysis
  technique there, one can achieve its global convergence. Moreover,
  by the strict saddle property of $\Phi_{\lambda}$ established in \cite{Li18}
  and the equivalence relation between \eqref{Nuclear-norm} and \eqref{Phi-lam}
  (see Lemma \ref{lemma2} in Appendix C), if the parameter $\lambda$
  is chosen such that \eqref{Nuclear-norm} has an optimal solution with rank at most $r$,
  then the sequence generated by Algorithm \ref{AAL} with $r\ge r^*$ for the least squares loss
  associated to a full sampling operator very likely converges to a global optimal solution
  $(\overline{U},\overline{V})$ of \eqref{Phi-lam} in a linear rate,
  and the error bound of $\overline{X}=\overline{U}\overline{V}^{\mathbb{T}}$
  to the true $M^*$ is $O(\sqrt{r^*}(\lambda+\|\mathcal{A}^*(\omega)\|))$.

  \medskip
  \noindent
  {\bf(ii)} When the parameter $\beta_k$ is chosen by the formula $\beta_k=\frac{\theta_{k-1}-1}{\theta_k}$
  with $\theta_k$ updated by
  \[
    \theta_{k+1}:=\frac{1}{2}\Big(1+\!\sqrt{1+4\theta_k^2}\Big)\ \ {\rm for}\ \
    \theta_{-1}=\theta_{0}=1,
  \]
  the accelerated strategy in Algorithm \ref{AAL} is Nesterov's
  extrapolation technique \cite{Nesterov83}. Unless otherwise stated,
  all numerical results are computed by this accelerated strategy.

    \medskip
  \noindent
  {\bf(iii)} By comparing the optimal conditions of the two subproblems with
  that of \eqref{Phi-lam}, when
   \begin{subnumcases}{}\label{cond1}
     \frac{\|\nabla_{1}F(\widetilde{U}^k,{V}^k)-\nabla_{1}F(U^{k+1},{V}^{k+1})
    +L_F(U^{k+1}-\widetilde{U}^k)\|}{1+\|y\|}\le \epsilon,\\
    \label{cond2}
   \frac{\|\nabla_{2}F(U^{k+1},\widetilde{V}^{k})-\nabla_{2}F(U^{k+1},{V}^{k+1})
    +L_F(V^{k+1}-\widetilde{V}^k)\|}{1+\|y\|}\le \epsilon
  \end{subnumcases}
  holds for a pre-given tolerance $\epsilon>0$, we terminate Algorithm \ref{AAL}
  at the iterate $(U^{k+1},V^{k+1})$.
 \end{remark}

 We take the least squares loss \eqref{Lsquare} for example to confirm
 our theoretical results. For the subsequent testing, the starting point
 $(U^0,V^0)$ of Algorithm \ref{AAL} is always chosen as
 $(P{\rm Diag}([\sigma^{r}(X^0)]^{1/2}), Q{\rm Diag}([\sigma^{r}(X^0)]^{1/2}))$
 with $(P,Q)\in\mathbb{O}^{n,m}(X^0)$ for $X^0=\mathcal{A}^*(y)$. It should be emphasized
 that such a starting point is not close to the bi-factors of $M^*$ unless
 $r=r^*$. All numerical tests are done by a desktop computer running on
 64-bit Windows Operating System with an Intel(R) Core(TM) i7-7700 CPU 3.6GHz
  and 16 GB memory.
 \subsection{RMSE comparison with convex relaxation method}\label{sec5.2}

  We compare the relative RMSE (root-mean-square-error) of the output of
  Algorithm \ref{AAL} for solving \eqref{Phi-lam} with that of the optimal solution
  yielded by the accelerated proximal gradient (APG) method for solving \eqref{Nuclear-norm}
  (see \cite{Toh2010}). The RMSE is defined as
  \[
    {\rm RMSE}:={\|X^f-M^*\|_F}/{\|M^*\|_F}.
  \]
  where $X^f$ represents the final output of a solver.
  We generate the vector $y\in\mathbb{R}^p$ via model \eqref{observe-model},
  where the true $M^*$ is generated by $M^*=U^*(V^*)^\mathbb{T}$ with $U^*\in\mathbb{R}^{n\times r^*}$
  and $V^*\in\mathbb{R}^{m\times r^*}$, the sampling operator $\mathcal{A}$
  is defined by $(\mathcal{A}(X))_i=\langle A_i,X\rangle$ for $i=1,2,\ldots,p$
  with $A_1,\ldots,A_p$ being i.i.d random Gaussian matrix whose entries
  follow the normal distribution $N(0,\frac{1}{p})$,
  and the entries of $\omega$ are i.i.d. and follow the normal distribution
  $N(0,\sigma_\omega^2)$ with $\sigma_\omega={0.1\|\mathcal{A}(M^*)\|}/{\|\xi\|}$
  for $\xi\thicksim N(0,I_p)$.

  \medskip

  We take $n=m=100, r^*=5$ and $p=1950$ for testing. Figure \ref{fig1} plots
  the relative RMSE of Algorithm \ref{AAL} for
  solving \eqref{Phi-lam} with $r=3r^*$ and
  $\lambda=\nu\|\mathcal{A}^*(\omega)\|$ and that of APG for solving
  the convex problem \eqref{Nuclear-norm} with the same $\lambda$.
  The stopping tolerance for the two solvers is chosen as $10^{-5}$.
  For each $\nu$, we conduct $5$ tests and calculate the average relative
  RMSE of the total tests. We see that the relative RMSE of two solvers
  has very little difference, but for $\lambda\le 1.5\|\mathcal{A}^*(\omega)\|$
  the outputs of Algorithm \ref{AAL} have lower ranks. This is not only
  consistent with the discussion in Remark \ref{main-remark1} (iii) but also
  implies that the factorization approach yields a lower rank solution
  with the same relative error.
 \begin{figure}[h]
 \begin{center}
 {\includegraphics[width=1.0\textwidth]{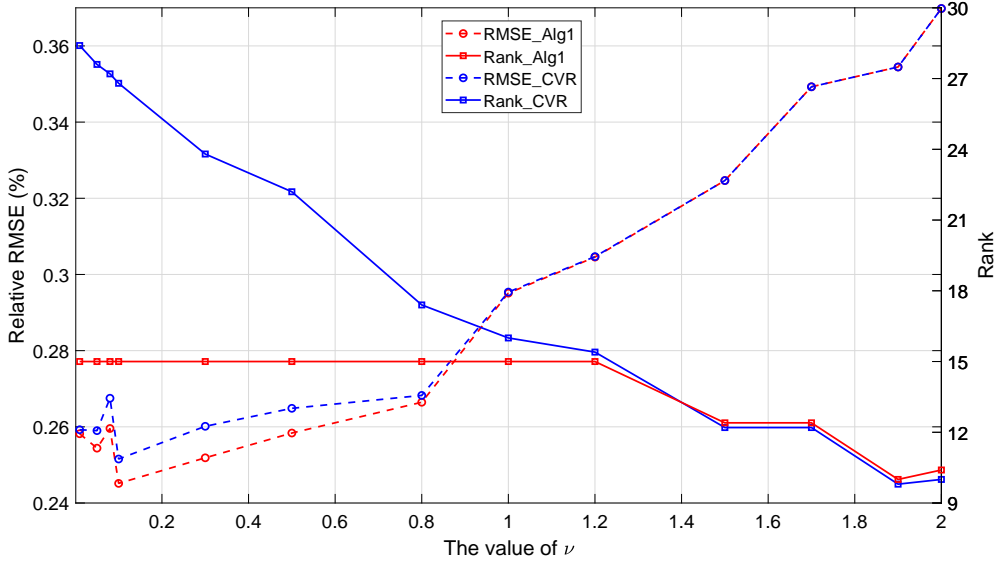}}\\  
 \caption{\small The average relative RMSE and rank of two solvers under different $\lambda$}
 \label{fig1}
 \end{center}
 \end{figure}
 \subsection{Illustration of linear convergence}\label{sec5.3}

  We take a matrix completion problem for example to illustrate the linear
  convergence of Algorithm \ref{AAL} without accelerated strategy when
  solving problem \eqref{Phi-lam} under the noisy and full sample setting,
  i.e., $f(X)=\frac{1}{2}\|X-M\|^2$ for $X\!\in\mathbb{R}^{n\times m}$,
  where $M=M^*+E$ is a noisy observation.
  The true $M^*\in\mathbb{R}^{n\times m}$ is generated in the same way
  as in Subsection \ref{sec5.2}, and the noise matrix $E$ is randomly generated
  by $E=\frac{0.1\|M^*\|}{\|B\|_F}B$, where every entry of $B$ obeys the standard
  normal distribution $N(0,1)$. We take $n=m=3000$ and $r^*=15$.
 \begin{figure}[h]
 \begin{center}
 {\includegraphics[width=12cm,height=6.5cm]{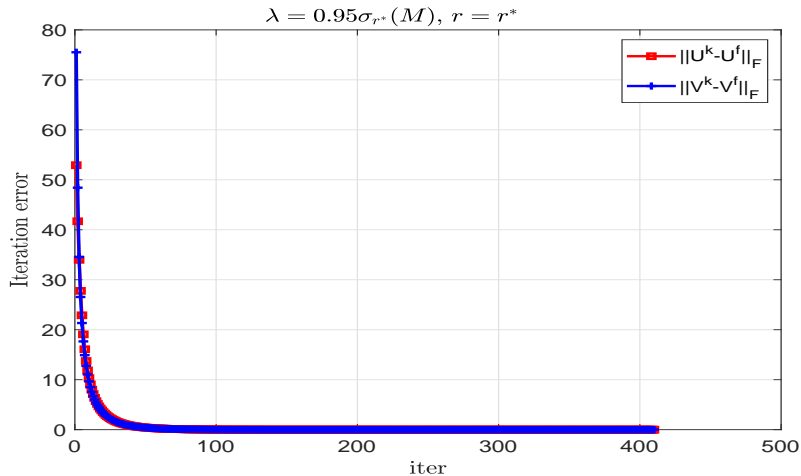}}\\  
 \caption{\small The iteration errors of Algorithm \ref{AAL} for minimizing the function $\Phi_{\lambda}$}
 \label{fig2}
 \end{center}
 \end{figure}

 \medskip

 Figure \ref{fig2} plots the iteration error curve $\|(U^k,V^k)-(U^f,V^f)\|_F$
 of Algorithm \ref{AAL}, where $(U^f,V^f)$ is the final output of Algorithm \ref{AAL}
 for solving \eqref{Phi-lam} with $r=r^*, \epsilon=10^{-10}$ and the number
 of max iteration $k_{\rm max}=5000$. We see that the sequence $\{(U^k,V^k)\}$
 displays the linear convergence behavior. By the strict saddle property in \cite{Li18},
 there is a high probability for the limit of $\{(U^k,V^k)\}$ to be a global
 optimal solution of problem \eqref{Phi-lam} if $\lambda=0.95\sigma_{r^*}(M)$ is such that \eqref{Nuclear-norm}
 has an optimal solution with rank at most $r^*$. Consequently, the convergence behavior is consistent
 with the result of Theorem \ref{KL-Philam}.

 \section{Conclusion}\label{sec6}

  For the factorization form \eqref{Phi-lam} of the nuclear norm regularized problem,
  under a restricted condition number assumption on $\nabla^2f$,
  we have derived the error bound to the true $M^*$ for the non-strict
  critical points with rank at most $r^*$, which is demonstrated to
  be optimal in the ideal noiseless and full sampling setup. Furthermore,
  in the noisy and full sampling setup we have established the KL property
  of exponent $1/2$ of its objective function $\Phi_{\lambda}$ associated to almost
  all $\lambda>0$ in the global minimizer set, and in the noisy and partial sampling setup,
  have also achieved this property of $\Phi_{\lambda}$ only at a class of stationary points.
  This result, along with the strict saddle property in \cite{Li18}, partly improves
  the convergence analysis result of some first-order methods for problem \eqref{Phi-lam}
  such as the alternating minimization methods in \cite{Recht10,Hastie15}.
  It is interesting to consider the error bound of critical points for other
  equivalent or relaxed factorization form of the rank regularized model \eqref{rank-reg}.
  We will leave them as our future research topics.

 \bigskip
 \noindent
 {\large\bf Acknowledgements.} The authors are deeply indebted to Professor
 Bhojanapalli from Toyota Technological Institute at Chicago for providing
 some helpful comments on \cite[Lemma 4.4]{Bhojanapalli16}. The research of
 Shaohua Pan is supported by the National Natural Science Foundation of China
 under project No.11971177 and Guangdong Basic and Applied Basic Research Foundation
(2020A1515010408).


  \bigskip
  \noindent
  {\bf\large Appendix A: The proof of Lemma \ref{fcond-asym}.}\\

  \begin{aproof}
  Fix any critical point $(U,V)$ of $\Phi_{\lambda}$. Then, for any
  $(Z_U,Z_V)\in\mathbb{R}^{n\times r}\times\mathbb{R}^{m\times r}$, we have
  \[
   \langle \nabla\Phi_{\lambda}(U,V),Z\rangle=0\ \ {\rm  with}\ Z=(Z_U;Z_V).
  \]
  Recall that $M^*=U^*{V^*}^{\mathbb{T}}$ with $(U^*,V^*)\in\mathcal{E}^*$.
  From \eqref{gradient}, this equality is equivalent to
  \begin{align}\label{equa-31}
   &\Big\langle \left(\begin{array}{cc}
                 \textbf{ 0} & \nabla\!f(X)-\nabla\!f(M^*) \\
                  \nabla\!f(X)^\mathbb{T}-\nabla\!f(M^*)^\mathbb{T} & \textbf{0} \\
                  \end{array}
            \right),ZW^\mathbb{T} \Big\rangle
    +\langle \Xi(M^*),ZW^\mathbb{T}\rangle=0\nonumber\\
   &\Longleftrightarrow
    \langle \nabla\!f(X)-\nabla\!f(M^*),
            Z_U V^\mathbb{T}+ U Z_V^\mathbb{T} \rangle
    +\langle \Xi(M^*),ZW^\mathbb{T}\rangle=0\nonumber\\
   &\Longleftrightarrow
   \int_0^1 \nabla^2f(tX+(1-t)M^*)(X-M^*,
            Z_U V^\mathbb{T}+ U Z_V^\mathbb{T})dt
    +\langle \Xi(M^*),ZW^\mathbb{T}\rangle=0
  \end{align}
  where $X=UV^{\mathbb{T}}$. Since ${\rm rank}([X\!-M^*\ \ Z_U V^\mathbb{T}+ U Z_V^\mathbb{T}])\le 4r^*$,
  by Lemma \ref{RSC-RSS-ineq} we have
  \begin{align*}
  \Big|\frac{2}{\alpha+\beta}\nabla^2f(tX+(1-t)M^*)(X\!-\!M^*,
            Z_U V^\mathbb{T}+ U Z_V^\mathbb{T})-
  \langle X\!-\!M^*, Z_U V^\mathbb{T}+ U Z_V^\mathbb{T}\rangle\Big|\\
  \leq \frac{\beta-\alpha}{\alpha+\beta}\big\|UV^\mathbb{T}-U^*{V^*}^\mathbb{T}\big\|_F
  \big\|Z_U V^\mathbb{T}+UZ_V^\mathbb{T}\big\|_F.
  \end{align*}
  By combining this inequality with equation \eqref{equa-31}, it follows that
  \begin{align}\label{I1I2I3}
  \Big|\underbrace{\frac{2}{\alpha+\beta}\langle \Xi(M^*),ZW^\mathbb{T}\rangle}_{I_1}
    +\underbrace{\langle UV^\mathbb{T}-U^*(V^*)^\mathbb{T},Z_U V^\mathbb{T}+UZ_V^\mathbb{T}\rangle}_{I_2}\Big|\nonumber\\
  \le \frac{\beta-\alpha}{\alpha+\beta}\| UV^\mathbb{T}-U^*(V^*)^\mathbb{T}\|_F
      \underbrace{\|Z_U V^\mathbb{T}+UZ_V^\mathbb{T}\|_F}_{I_3}.
  \end{align}
  Now take $Z=(WW^\mathbb{T}-W^*(W^*)^\mathbb{T})(W^{\mathbb{T}})^{\dagger}$
  where $W^*$ is defined as in \eqref{W-UV} with $(U^*,V^*)$.
  Since the column orthonormal matrix $Q$ spans the subspace ${\rm col}(W)$,
  it is not hard to check that $(W^{\mathbb{T}})^{\dagger}W^{\mathbb{T}}=QQ^{\mathbb{T}}$.
  Then, it follows that
  $ZW^\mathbb{T}=(WW^\mathbb{T}-W^*(W^*)^\mathbb{T})QQ^\mathbb{T}$.
  Next we bound the terms $I_1, I_2$ and $I_3$ successively.
  First, for the term $I_1$, it holds that
  \begin{equation}\label{I1-ineq}
    I_1=\frac{2}{\alpha+\beta}\langle \Xi(M^*),ZW^\mathbb{T}\rangle
      =\frac{2}{\alpha+\beta}\langle\Xi(M^*),(WW^\mathbb{T}-W^*(W^*)^\mathbb{T})QQ^\mathbb{T}\rangle.
  \end{equation}
  For the term $I_2$, by recalling the definition of the linear operator
  $\mathcal{P}_{\rm off}$, we have
  \[
     I_2 =\langle UV^\mathbb{T}-U^*(V^*)^\mathbb{T},Z_U V^\mathbb{T}+UZ_V^\mathbb{T}\rangle
     =\langle \mathcal{P}_{\rm off} (WW^\mathbb{T}-W^*(W^*)^\mathbb{T}), ZW^\mathbb{T} \rangle.
  \]
  By the expressions of $W,W^*$ and $\widehat{W},\widehat{W}^*$, it is not hard to check that
  \[
    \mathcal{P}_{\rm off}(WW^\mathbb{T}-W^*(W^*)^\mathbb{T})
    =\frac{1}{2}\big[WW^\mathbb{T}-W^*(W^*)^\mathbb{T}-\widehat{W}\widehat{W}^\mathbb{T}
                      +\widehat{W}^*(\widehat{W}^*)^\mathbb{T}\big],
  \]
  which along with $ZW^\mathbb{T}=(WW^\mathbb{T}-W^*(W^*)^\mathbb{T})QQ^\mathbb{T}$ implies that
  \begin{align*}
    I_2&=\frac{1}{2}\langle WW^\mathbb{T}-W^*(W^*)^\mathbb{T},(WW^\mathbb{T}-W^*(W^*)^\mathbb{T})QQ^\mathbb{T}\rangle\\
    &\quad- \frac{1}{2}\langle \widehat{W}\widehat{W}^\mathbb{T}-\widehat{W}^*(\widehat{W}^*)^\mathbb{T},
     (WW^\mathbb{T}-W^*(W^*)^\mathbb{T})QQ^\mathbb{T} \rangle.
  \end{align*}
  Since $(U,V)$ is a stationary point of $\Phi_{\lambda}$, from Lemma \ref{stationary-lemma1}
  and $(U^*,V^*)\in\mathcal{E}^*$ we have
  \begin{align*}
   &\langle \widehat{W}\widehat{W}^\mathbb{T}-\widehat{W}^*(\widehat{W}^*)^\mathbb{T},
     (WW^\mathbb{T}-W^*(W^*)^\mathbb{T})QQ^\mathbb{T} \rangle\\
   &=-\langle \widehat{W}\widehat{W}^\mathbb{T},
      W^*(W^*)^\mathbb{T}QQ^\mathbb{T}\rangle-
     \langle\widehat{W}^*(\widehat{W}^*)^\mathbb{T}, WW^\mathbb{T}QQ^\mathbb{T}\rangle\\
     &=-\langle\widehat{W}^*(\widehat{W}^*)^\mathbb{T}, WW^\mathbb{T}\rangle\le 0
  \end{align*}
  where the second equality is using $QQ^{\mathbb{T}}=(W^{\mathbb{T}})^{\dagger}W^{\mathbb{T}}$,
  $\widehat{W}^\mathbb{T}W=0$ and $(\widehat{W}^*)^\mathbb{T}W^*=0$,
  and the inequality is due to the positive semidefiniteness of $\widehat{W}^*(\widehat{W}^*)^\mathbb{T}$
  and $WW^\mathbb{T}$. Thus,
  \begin{equation}\label{I2-ineq}
    I_2\ge \frac{1}{2}\|(WW^\mathbb{T}-W^*(W^*)^\mathbb{T})Q\|_F^2
    =\frac{1}{2}\|(WW^\mathbb{T}-W^*(W^*)^\mathbb{T})QQ^{\mathbb{T}}\|_F^2.
  \end{equation}
  For the term $I_3$, by recalling that $Z=(WW^\mathbb{T}-W^*(W^*)^\mathbb{T})(W^{\mathbb{T}})^{\dagger}$,
  we calculate that
 \begin{align}\label{I3-ineq}
  I_3&=\|Z_U V^\mathbb{T}+UZ_V^\mathbb{T}\|_F\leq\sqrt{2}\|\mathcal{P}_{\rm off}(ZW^\mathbb{T})\|_F
      =\|ZW^{\mathbb{T}}\|_F\nonumber\\
     &=\|(WW^\mathbb{T}-W^*(W^*)^\mathbb{T})QQ^\mathbb{T}\|_F.
 \end{align}
 Now combining equation \eqref{I1-ineq}-\eqref{I3-ineq} with \eqref{I1I2I3} yields the desired result.
 \end{aproof}

 \medskip
 \noindent
 {\bf\large Appendix B: The proof of Lemma \ref{scond-asym}.}\\

  \begin{aproof}
  Since $\Delta\Delta^\mathbb{T}\!=\!WW^\mathbb{T}-W(W^*R_1^*)^\mathbb{T}-W^*R_1^*W^\mathbb{T}+W^*(W^*)^\mathbb{T}$
  where $R_1^*$ is the matrix consisting of the first $r^*$ rows of $R^*$,
  with $X=UV^{\mathbb{T}}$ it follows that
  \begin{align*}
   &2\langle\nabla\!f(X),\Delta_U \Delta_V^\mathbb{T} \rangle+\lambda\langle \Delta,\Delta\rangle
    =\langle\Xi(X),\Delta\Delta^\mathbb{T}\rangle\\
   &=\big\langle\Xi(X),WW^\mathbb{T}-W(W^*R_1^*)^\mathbb{T}-W^*R_1^*W^\mathbb{T}+W^*(W^*)^\mathbb{T}\big\rangle\\
   &=\big\langle\Xi(X),W^*(W^*)^\mathbb{T}\big\rangle=\big\langle\Xi(X),W^*(W^*)^\mathbb{T}-WW^\mathbb{T}\big\rangle
  \end{align*}
  where the third equality is due to $\Xi(X)W=0$ and $\Xi(X)=\Xi(X)^{\mathbb{T}}$.
  Together with \eqref{Hessian}, we get \eqref{goal-equa}.
  We next show that \eqref{goal-ineq} holds.
  From $\nabla^2\Phi_{\lambda}(U,V)(\Delta,\Delta)\ge 0$ and \eqref{goal-equa},
  \begin{equation}\label{mid-ineq1}
    \nabla^2f(X)(U\Delta_V^\mathbb{T}+\Delta_U V^\mathbb{T},U\Delta_V^\mathbb{T}+\Delta_U V^\mathbb{T})
    \ge\langle\Xi(X),WW^\mathbb{T}-W^*(W^*)^\mathbb{T}\rangle.
  \end{equation}
  According to the given assumption on $f$, it is immediate to have that
  \begin{align}\label{mid-ineq2}
   \nabla^2f(X)(U\Delta_V^\mathbb{T}+\Delta_U V^\mathbb{T},U\Delta_V^\mathbb{T}+\Delta_U V^\mathbb{T})
   &\le \beta\|U\Delta_V^\mathbb{T}+\Delta_UV^\mathbb{T}\|_F^2,\nonumber\\
   &\le 2\beta(\|U\Delta_V^\mathbb{T}\|^2_F+\|\Delta_UV^\mathbb{T}\|_F^2).
  \end{align}
  In addition, from the restricted strong convexity of $f$, it follows that
  \begin{align*}
   &\langle\Xi(X),WW^\mathbb{T}-W^*(W^*)^\mathbb{T}\rangle\nonumber\\
   &=\langle\Xi(X)-\Xi(M^*),WW^\mathbb{T}-W^*(W^*)^\mathbb{T}\rangle +\langle \Xi(M^*),WW^\mathbb{T}\!-\!W^*(W^*)^\mathbb{T}\rangle\nonumber\\
   &=2\langle\nabla f(X)\!-\!\nabla f(M^*),X\!-\!M^*\rangle +\langle \Xi(M^*),WW^\mathbb{T}\!-\!W^*(W^*)^\mathbb{T}\rangle\nonumber\\
   &=2\int_0^1\nabla^2\!f(M^*+t(X\!-\!M^*))(X\!-\!M^*,X\!-\!M^*)dt
     +\langle \Xi(M^*),WW^\mathbb{T}-W^*(W^*)^\mathbb{T}\rangle\nonumber\\
   &\geq 2\alpha \|X-M^*\|_F^2+\langle \Xi(M^*),WW^\mathbb{T}\!-\!W^*(W^*)^\mathbb{T}\rangle.
  \end{align*}
  Together with inequalities \eqref{mid-ineq1} and \eqref{mid-ineq2},
  we obtain
 \begin{align}\label{temp-WDeta}
  &2\beta(\|U\Delta_V^\mathbb{T}\|^2_F+\|\Delta_UV^\mathbb{T}\|_F^2)
  \geq 2\alpha \|X-M^*\|_F^2+\langle \Xi(M^*),WW^\mathbb{T}-W^*(W^*)^\mathbb{T}\rangle\nonumber\\
  &\Longleftrightarrow
   \beta\|W\Delta^\mathbb{T}\|_F^2\ge
   2\alpha \|X-M^*\|_F^2+\langle \Xi(M^*),WW^\mathbb{T}-W^*(W^*)^\mathbb{T}\rangle
 \end{align}
 where the equivalence is due to $\|U\Delta_V^\mathbb{T}\|^2_F=\|V\Delta_V^\mathbb{T}\|_F^2$
 and $\|V\Delta_U^\mathbb{T}\|^2_F=\|U\Delta_U^\mathbb{T}\|_F^2$, implied by
 $U^{\mathbb{T}}U=V^{\mathbb{T}}V$. From \cite[Lemma 4.5]{Li18} it follows that
 \begin{align*}
  \|WW^\mathbb{T}-W^*(W^*)^\mathbb{T}\|^2_F
  &=\|\mathcal{P}_{\rm on}(WW^\mathbb{T}-W^*(W^*)^\mathbb{T})\|^2_F
  +\|\mathcal{P}_{\rm off}(WW^\mathbb{T}-W^*(W^*)^\mathbb{T})\|^2_F\\
  &\le 2\|\mathcal{P}_{\rm off}(WW^\mathbb{T}-W^*(W^*)^\mathbb{T})\|^2_F= 4\|X-M^*\|_F^2,
\end{align*}
 which implies that $2\alpha \|X-M^*\|_F^2\geq \frac{\alpha}{2}\|WW^\mathbb{T}-W^*(W^*)^\mathbb{T}\|^2_F$.
 Together with \eqref{temp-WDeta}, we obtain the desired inequality \eqref{goal-ineq}. The proof is completed.
 \end{aproof}

 \medskip
 \noindent
 {\bf\large Appendix C:}

 \medskip

 The following lemma states the relation between the optimal solution set
 of \eqref{Nuclear-norm} and the global optimal solution set of \eqref{Phi-lam},
 whose proof is easy by the following result in \cite{RSebro05}:
 \begin{equation}\label{nuclea-char}
    \|X\|_*=\min_{R\in\mathbb{R}^{n\times r},L\in\mathbb{R}^{m\times r}}
    \Big\{\frac{1}{2}\big(\|R\|_F^2+\|L\|_F^2\big)\ \ {\rm s.t.}\ \ X=RL^{\mathbb{T}}\Big\}.
 \end{equation}
 \vspace{-0.5cm}
 \begin{alemma}\label{lemma2}
  Fix any $\lambda>0$. If $(\overline{U},\overline{V})$ is globally optimal
  to problem \eqref{Phi-lam}, then $\overline{X}=\overline{U}{\overline{V}}^{\mathbb{T}}$
  is an optimal solution of \eqref{Nuclear-norm} over the set
  $\{X\in\mathbb{R}^{n\times m}\ |\ {\rm rank}(X)\le r\}$; and conversely,
  if $\overline{X}$ is an optimal solution of \eqref{Nuclear-norm} with
  ${\rm rank}(\overline{X})\le r$, then $(\overline{R},\overline{L})$
  with $\overline{R}=\overline{P}[{\rm Diag}(\sigma^{r}(\overline{X}))]^{1/2}$
  and $\overline{L}=\overline{Q}[{\rm Diag}(\sigma^{r}(\overline{X}))]^{1/2}$
  for $(\overline{P},\overline{Q})\in\mathbb{O}^{n,m}(\overline{X})$
  is a global optimal solution to \eqref{Phi-lam}.
 \end{alemma}

 \end{document}